\documentclass[11pt]{article}
\usepackage{amsthm,amsfonts, amsbsy, amssymb,amsmath,graphicx}
\usepackage{graphics}
\usepackage[english]{babel}
\usepackage{graphicx}
\usepackage{subcaption}
\usepackage{lineno}
\usepackage{enumerate}
\usepackage{tikz}
\usepackage{tkz-graph}
\usepackage{tkz-berge}
\usepackage{hyperref}
\usepackage{color}
\usepackage{tikz}
\usetikzlibrary{arrows, automata}

\hypersetup{colorlinks=true}

\hypersetup{colorlinks=true, linkcolor=blue, citecolor=blue,urlcolor=blue}
\title{Paired coalition in graphs}
\author{{\small  Mohammad Reza Samadzadeh$^{a}$\thanks{Corresponding author}, Doost Ali Mojdeh$^{b}$\thanks{Corresponding author}, Reza Nadimi$^{c}$}\\ {\small $^{a,b}$Department of
		Mathematics, Faculty of Mathematical Sciences}\\{\small University of Mazandaran, Babolsar, Iran}\\
 {\small $^{c}$Department of
		Computer science, Faculty of Mathematical Sciences}\\{\small University of Mazandaran, Babolsar, Iran}\\
	{\small $^a$m.samadzadeh02@umail.umz.ac.ir}\\ {\small $^b$damojdeh@umz.ac.ir} \\{\small $^c$nadimi@umz.ac.ir}}
\date{}
\newtheorem{theorem}{Theorem}[section]
\newtheorem{corollary}[theorem]{Corollary}
\newtheorem{lemma}[theorem]{Lemma}

\newtheorem{observation}[theorem]{Observation}
\newtheorem{proposition}[theorem]{Proposition}

\newtheorem{p}{Problem}
\newtheorem{con}{Conjecture}
\theoremstyle{definition}
\newtheorem{definition}[theorem]{Definition}
\theoremstyle{remark}
\newtheorem{rem}[theorem]{Remark}
\begin{document}
	
	\maketitle
	\begin{abstract}
		
		\noindent A paired  coalition in a graph $G=(V,E)$ consists of two disjoint sets of vertices $V_1$ and $V_2$, neither of which is a paired dominating set but whose union $V_1 \cup V_2$ is a paired dominating set. A paired coalition partition (abbreviated $pc$-partition) in a graph $G$ is a vertex partition $\pi= \lbrace V_1,V_2,\dots ,V_k \rbrace$ such that each set $V_i$ of $\pi$  is not a paired dominating set but forms a paired coalition with another set $V_j \in \pi$. The paired coalition graph $PCG(G,\pi) $ of the graph $G$ with the $pc$-partition $\pi$ of $G$, is the graph whose vertices correspond to the sets of $\pi$, and two vertices $V_i$ and $V_j$ are adjacent in  $PCG(G,\pi) $ if and only if their corresponding sets $V_i$ and $V_j$ form a paired coalition in $G$. In this paper, we  initiate the  study of  paired coalition partitions and paired coalition graphs. In particular, we determine the paired coalition number of paths and cycles, obtain some results on paired coalition partitions in trees and   characterize pair coalition graphs of paths, cycles and trees. We also characterize  triangle-free graphs $G$ of order $n$ with $PC(G)=n$ and unicyclic graphs $G$ of order $n$ with $PC(G)=n-2$.   	
\end{abstract}

	\vspace{1mm}
	\section{Introduction}
	All graphs considered in this paper are simple, finite and undirected.	Let $G=(V,E)$ denote a   graph of order $n$ with vertex set $V=V(G)$ and edge set $E=E(G)$. The {\em open
		neighborhood} of a vertex $v\in V$ is the set $N(v)=\lbrace u \lvert \lbrace u,v\rbrace \in E \rbrace$, and its {\em closed
		neighborhood} is the set $N[v]=N(v) \cup \lbrace v \rbrace$. Each vertex of $N(v)$ is called a {\em neighbor}
	of $v$, and  the cardinality of $N(v)$ is called the {\em degree} of $v$, denoted by $deg(v)$ or $deg_G (v)$. A
	vertex $v$ of degree $1$ is called a {\em pendant vertex} or {\em leaf}, and its neighbor is called a {\em support vertex}.  A vertex of degree $n-1$ is called a {\em full vertex} while a vertex of degree $0$ is called an {\em isolated vertex}. The {\em minimum} and the {\em maximum
		degree} of $G$ is denoted by $\delta (G)$ and $\Delta (G)$, respectively.  For a set $S$ of vertices of $G$, the subgraph induced by $S$ is denoted by  $G[S]$.
	For two  sets $X$ and $Y$ of vertices, let  $[X,Y]$ denote the set of edges between $X$ and $Y$.	
	If every vertex of $X$ is adjacent to every vertex of  $Y$, we say that $[X,Y]$ is {\em full}, while
	if there are no edges between them, we say that $[X,Y]$ is {\em empty}.
	A subset $V_i \subseteq V$ is called a {\em singleton set} if $\lvert V_i \rvert =1$, and is called a {\em non-singleton set} if $\lvert V_i \rvert \geq 2$.

	 We denote the  path, cycle, complete graph and	
	star of order $n$ by $P_n$, $C_n$, $K_n$ and $K_{1,n-1}$, respectively.
	A double star
	with respectively $p$ and $q$ leaves adjacent to each support vertex is denoted by $S_{p,q}$. If $G$ has a $u,v$-path, then the distance from $u$ to $v$, is the length of a shortest $u,v$-path. The {\em girth} of a graph $G$ is the length of its shortest cycle and is denoted by $g(G)$. The complete graph $K_3$ is called a {\em triangle}, and a graph is {\em triangle-free}
	if it has no $K_3$ as an induced subgraph  \cite{ref13}.
	In a tree $T$, a vertex is called a {\em strong support vertex} if it has at least two leaf neighbors.  A graph $G$ is unicyclic if it has exactly one cycle. 
	
	A set $S\subseteq V$  in a graph $G=(V,E)$ is called a {\em dominating set} if every vertex $v\in V$ is either an element of $S$ or is adjacent to an element of $S$. The minimum cardinality of a dominating set of $G$ is called domination number of $G$, denoted by $\gamma (G)$ \cite{ref10}. A {\em paired dominating set} of a graph $G = (V ,E)$
is a dominating set $S \subseteq V$ such that the induced subgraph $G[S]$ contains a perfect matching. 
The minimum cardinality of a paired dominating set of $G$ is called paired domination number of $G$, denoted by $\gamma_{pr} (G)$ \cite{ref14}.  
	
	The term {\em coalition} was introduced by   Haynes et al, \cite{ref6} and has been studied further in \cite{ref3,ref4,ref5,ref7,ref8} and elsewhere.  Some  variants of this concept have been investigated in papers such as \cite{ref1,ref9,ref2}. Alikhani et al, studied the concepts of {\em total coalition} and {\em connected coalition}  in papers \cite{ref1} and \cite{ref9}. Furthermore, The concept of {\em independent coalition} has been studied in \cite{ref2}.    
	In this paper, we inrtroduce the concept of {\em paired coalition}  and initiate the study of this concept. In what follows, we define the terms paired coalition, {\em paired coalition partition} and {\em paired coalition graph}.

	\begin{definition}\label{pc-def}
		Let $G=(V,E)$ be a graph. A paired coalition in a graph $G$ consists of  two disjoint sets of  vertices $V_1$ and $V_2$, neither of which is a paired dominating set but whose union $V_1\cup V_2$ is a paired dominating set. We say the sets $V_1$ and $V_2$ form a paired coalition, and are
		$pc$-partners.
		
	\end{definition}
	\begin{definition}\label{pcp-def}
		A paired coalition partition, abbreviated $pc$-partition, in a graph $G$ is a vertex partition
		$\pi=\{V_1,V_2,\dots ,V_k\}$ such that  every set $V_i$ of $\pi$ is not a paired dominating
		set but forms a paired coalition with another set $V_j$ in $\pi$.  The paired coalition number $PC(G)$ equals the maximum order $k$ of
		a $pc$-partition of $G$, and a $pc$-partition of $G$ having order $PC(G)$ is called a $PC(G)$-partition.
	\end{definition}
As we will see later, not every graph $G$ admits a $pc$-partition. In such situations we will say that $PC(G)=0$.
\begin{definition}
	Given a graph $G$ with a $pc$-partition $\pi =\lbrace V_1,V_2,\dots ,V_k\rbrace$ of $G$,
	the paired coalition graph $PCG(G,\pi)$ is the graph with $k$ vertices labeled $\lbrace V_1,V_2,\dots ,V_k\rbrace$,
	corresponding one-to-one with the elements of $\pi$, and two vertices $V_i$ and $V_j$ are
	adjacent in $PCG(G,\pi)$ if and only if the sets $V_i$ and $V_j$ are paired coalition partners in $\pi$,
	that is, neither $V_i$ nor $V_j$ is a paired dominating set of $G$, but $V_i \cup V_j$ is a paired dominating
	set of $G$.
\end{definition}
We will denote the paired coalition graphs of paths, cycles and trees by $PCP$-graphs, $PCC$-graphs and $PCT$-graphs, respectively.

This paper is organized as follows. Section $2$ is devoted to some preliminary results most of which are used in subsequent sections. In Section $3$, we determine  paired coalition number of paths and cycles. In Section $4$, we discuss the concept of paired coalition in trees focusing on perfect binary trees and trees with large paired coalition number. In Section $5$, we characterize triangle-free graphs  of order $n$ with paired coalition number $n$, and unicyclic graphs of order $n$ with paired coalition number $n-2$. Finally, we close the paper with some research problems.

\section{Preliminaries}
The following observations are immediate.
\begin{observation} \label{obs-full} 
	If $G$ is a graph of order $n\geq 2$ with at least one full vertex, then $PC(G)=n$.
\end{observation}
	\begin{observation} \label{obs-multi} 
	If $G$ is a  complete multipartite graph of order $n$, then  $PC(G)=n$.
	\end{observation}
	\begin{lemma} \label{l5}
	Let $G$ be a  graph of order $n$ with  $\delta(G)=1$. Further, let $x$ be a vertex of degree $1$, and let $y$ be the support vertex of $x$. Then, $PC(G)=n$ if and only if $y$ is a full vertex.
\end{lemma}

\begin{proof}
	Let $V(G)=\lbrace x,y,v_1,v_2,\dots ,v_{n-2}\rbrace$.
	If $y$ is a full vertex, then Observation \ref{obs-full} implies that $PC(G)=n$. Conversely, assume that $PC(G)=n$. Consider the singleton partition $\pi_1 =\lbrace X,Y,V_1,V_2,\dots ,V_{n-2} \rbrace$, where $x\in X$, $y\in Y$, and $v_i \in V_i$, for each $i=1,2,\dots ,n-2$. Since $N(x)=\lbrace y\rbrace$, each set in $\pi_1 \setminus \lbrace X,Y\rbrace$ must have a $pc$-partner in $\lbrace X,Y\rbrace$. But $X$ can only form a paired coalition with $Y$, since for each $i=1,2,\dots ,n-2$, $G[X \cup V_i]$ has no perfect matching. Hence, each set in $\pi_1 \setminus \lbrace X,Y\rbrace$ must form a paired coalition with $Y$. Therefore, $y$ is a full vertex.
\end{proof}

\begin{lemma} \label{lem.sup1} 
	Let  $G $ be a graph with $\delta (G)=1$, and    the set of  support vertices $S$. If $\pi$ is a $pc$-partition of $G$, then for each pair $(A,B)$ of $pc$-partners in $G$, $S \subseteq A \cup B$.
\end{lemma}
\begin{proof}
	Suppose, to the contrary, that $S \nsubseteq A \cup B$, that is,  there exists a vertex $u\in S$ such that $u \notin A \cup B$. Let $v$ be the leaf attached to   $u$. if $v \notin A \cup B$, then $v$ is not dominated by $A \cup B$, a contradiction. Otherwise, $v$ is an isolated vertex in $G[A \cup B]$ which implies that $G[A\cup B]$ has no perfect matching, again a contradiction.
\end{proof}

\begin{lemma} \label{lem.sup2}
		Let $G $ be a graph with $\delta (G)=1$ and $PC(G)\geq 3$, and let $\pi$ be a $PC(G)$-partition. Then there exists a set in $\pi$  containing all support vertices of $G$.
\end{lemma}

\begin{proof}
 Let $S$ be the set of support vertices of $G$.	If $|S|=1$, then the result is obtained. Hence, we may assume that  $|S|\ge 2$. Suppose, to the contrary, that there are two vertices  $u,v$ in $S$ such that $u$ and $v$ are contained in different members of $\pi$. Let $u \in A$ and $v \in B$, where $\lbrace A,B\rbrace \subset \pi$. Further, let $C$ be a set in $\pi$, such that $C \notin \lbrace A,B\rbrace$. Now Lemma \ref{lem.sup1} implies that $C$ has no $pc$-partner, a contradiction.
\end{proof}

Using Lemmas \ref{lem.sup1} and \ref{lem.sup2}, we obtain the following result.
\begin{corollary} \label{coro-all}
	Let $G$ be a graph with $\delta (G)=1$ and $PC(G)\geq 3$, and let $\pi$ be a $PC(G)$-partition. Further, let $V_s$ be the member of $\pi$ containing support vertices of $G$. Then for any paired $(A,B)$ of $pc$-partners in $G$, $V_s \in \lbrace A,B\rbrace$.
\end{corollary}
\begin{corollary} \label{pcg-1}
Let $G$ be a graph with $\delta (G)=1$, and let $\pi$ be a $pc$-partition
of $G$ of order $k$. Then $PCG(G,\pi) \simeq K_{k-1,1}$.
\end{corollary}
\section{Paths and cycles}
In this section, we will determine the paired coalition number of  paths and cycles. We will also characterize the paired coalition graphs of paths and cycles.

\begin{theorem} \label{pc-path}  
	For the path $P_n$, 
	$$
	PC(P_n)= \begin{cases}
		0 & if \  n=1 \\
		2 & if \  n=2, 4 \\
		3 &   otherwise.
	\end{cases}
	$$ 
\end{theorem}

\begin{proof} 
By Definition \ref{pcp-def}, $PC(P_1)=0$ and $PC(P_2)=2$.  For the path $P_4 =(v_1,v_2,v_3,v_4)$,  it is easy to verify that $PC(P_4)\neq 4$, $PC(P_4)\neq 3$, and that the partition $\lbrace \lbrace v_1,v_2 \rbrace ,\lbrace v_3,v_4\rbrace \rbrace $  is a $pc$-partition of $P_4$. So, $PC(P_4)=2$.
	
	 Now consider the path $P_n$ with $V(P_n)=\lbrace v_1,v_2,\dots ,v_n\rbrace$ and $E(P_n)=\lbrace v_i v_{i+1} : 1\leq i\leq n-1\rbrace$, where $n \neq 1,2,4$.
	If $n \equiv 1 \pmod{2}$, then the partition $\lbrace \lbrace v_1\rbrace ,\lbrace v_2,\dots ,v_{n-1} \rbrace , \lbrace v_n \rbrace \rbrace$ is a $pc$-partition of $P_n$. Further, the partition $\lbrace \lbrace v_1,v_6\rbrace ,\lbrace v_2,v_5\rbrace ,\lbrace v_3,v_4\rbrace \rbrace$ is a $pc$-partition of $P_6$, and for the paths $P_n$, where $n\geq 8$, and $n \equiv 0 \pmod{2}$,  the partition $\lbrace \lbrace v_1,v_2,v_7,v_8,\dots ,v_n \rbrace ,\lbrace v_3,v_4\rbrace ,\lbrace v_5,v_6\rbrace \rbrace$ is a $pc$-partition of $P_n$. Hence, for each  $n\neq 1,2,4$,  $PC(P_n)\geq 3$.
	
	 To complete the proof, we show that $PC(P_n)\leq 3$, for each $n \geq 1$. The result is obvious for $n\leq 4$, so we may assume that $n\geq 5$. 
	Let $\pi$ be a $PC(P_n)$-partition. By Lemma \ref{lem.sup2}, there exists a member of $\pi$ (name $A$) that contains support vertices of $P_n$, that is, $\lbrace v_2,v_{n-1}\rbrace \subseteq A$. We show that $A$ forms a paired coalition with at most two other sets. 
	
	Suppose, to the contrary, that there exist three sets in $\pi$ (name $B$,$C$ and $D$) that  form a paired coalition with $A$. Let $v_i$, $v_j$, where $i < j$, be two vertices such that $v_i \in A$, $v_j \in A$ and for each $i < k < j$, $v_k \notin A$. First, we show that  $\lvert j-i\rvert \leq 3$. Suppose, to the contrary, that $\lvert j-i\rvert \geq 4$. Let $l$ be an arbitrary index such that $i+1 \leq l\leq j-3$.  Since $A \cup B$ is a dominating set, it follows that $B \cap \lbrace v_l, v_{l+1},v_{l+2}\rbrace \neq \emptyset$.
	
	 Using a similar argument, we have $C \cap \lbrace v_l, v_{l+1},v_{l+2}\rbrace \neq \emptyset$ and $D \cap \lbrace v_l, v_{l+1},v_{l+2}\rbrace \neq \emptyset$. Assume, without loss of generality, that $v_{l+1} \in B$.  Now $v_{l+1}$ is an isolated vertex in  $G[ A \cup B]$, implying that $G[A \cup B]$ does not have a perfect matching, a contradiction. Hence, $\lvert j-i\rvert \leq 3$. It follows that $A$ is a dominating set, and so has no perfect matching. Thus, $P_n [A]$ contains a connected component $C$ such that $\lvert V(C)\rvert \equiv 1 \pmod{2}$. Let $V(C)=\lbrace v_s,v_{s+1}, \dots ,v_t \rbrace$.  We show that $s\neq 1$. Suppose, to the contrary, that $s=1$. Let $v_{t+1} \in V_{t+1}$, where $V_{t+1} \in \pi$. Note that $A$ cannot be a $pc$-partner of any set in $\pi \setminus V_{t+1}$, and so by Corollary \ref{coro-all}, $PC(P_n) \leq 2$, a contradiction. Similarly, we can show that $t \neq n$. Now let $v_{s-1} \in V_{s-1}$ and $v_{t +1} \in V_{t+1}$, where $\lbrace V_{s-1},V_{t+1}\rbrace \subset \pi$. Note that $A$ cannot be a $pc$-partner of any set in $\pi \setminus \lbrace V_{s-1} ,V_{t+1}\rbrace$, which contradicts the hypothesis. Hence, $A$ forms a paired coalition with at most two other sets. Now  Corollary \ref{coro-all} implies that $PC(P_n)\leq 3$. This completes the proof.
\end{proof}

\begin{theorem}
	A graph $G$ is a $PCP$-graph if and only if $G \in \lbrace P_2,P_3\rbrace$.
\end{theorem}
\begin{proof}
	Consider the $PC(P_4)$-partition $\pi =\lbrace \lbrace v_1,v_2\rbrace, \lbrace v_3,v_4\rbrace \rbrace$, where $P_4 =(v_1,v_2,v_3,v_4)$.  We observe that $PCG(P_4,\pi) \simeq P_2$. Now consider  the $PC(P_5)$-partition $\pi = \lbrace \lbrace v_1\rbrace ,\lbrace v_2,v_3,v_4\rbrace ,\lbrace v_5\rbrace \rbrace$. We observe that $PCG(P_5,\pi)\simeq P_3$.  Conversely, let $G$ be a $PCP$-graph of order $n$. By Theorem \ref{pc-path}, we have $2 \leq  n \leq 3$. Now applying Corollary \ref{pcg-1}, we deduce that $G \in \lbrace P_2,P_3\rbrace$.
\end{proof}
\begin{theorem} \label{pc-cycle}
	For the cycle $C_n$,
	$$
	PC(C_n)= \begin{cases}
		4 & if \  n \equiv 0 \pmod{4} \\
		3 &    otherwise.
	\end{cases}
	$$
\end{theorem}
\begin{proof} 
First we note that indices here are taken modula $n$. Let $G=C_n$, where $V(G)=\lbrace v_1,v_2,\dots ,v_n\rbrace$ and $E(G)=\lbrace v_i v_{i+1} : 1\leq i\leq n\rbrace$, and let $\pi$ be a $PC(G)$-partition. Note that for any pair $(X,Y)$ of $pc$-partners in $G$,  $\lvert X\rvert +\lvert Y\rvert \geq \frac{n}{2}$. Now we consider two cases.

\textbf{Case 1.} $n \not\equiv 0 \pmod{4}$. If $n$ is odd, then the vertex partition $\lbrace \lbrace v_1\rbrace ,\lbrace v_2 \rbrace , \lbrace v_3,v_4, \dots ,v_n \rbrace \rbrace$, is a $pc$-partition of $G$. Otherwise, the vertex partition $\lbrace \lbrace v_1,v_2\rbrace ,\lbrace v_3,v_4 \rbrace , \lbrace v_5,v_6, \dots ,v_n \rbrace \rbrace$, is a $pc$-partition of $G$, so  $\lvert \pi \rvert \geq 3$. Now we show that $\lvert \pi \rvert \leq 4$. Let the sets $A$ and $B$ be $pc$-partners in $\pi$. Since $\lvert A \cup B \rvert \geq \frac{n+1}{2}$, it follows that for any  pair $(X,Y)$ of $pc$-partners,  $A \in \lbrace X,Y\rbrace$ or $B \in \lbrace X,Y\rbrace$. Let $\pi ^\prime =\lbrace C_1,C_2,\dots ,C_k \rbrace$  be the partition of $G[A]$  into its connected components. We  consider two subcases.

\textbf{Subcase 1.1.} $\lvert A\rvert  \equiv 1 \pmod{2}$. It follows that $\pi ^\prime$  contains at least one component with odd order. Let $C_i$  be such a component, with $V(C_i)=\lbrace v_s , v_{s+1}, \dots , v_t \rbrace$. Further, let $V_{s-1}$ and $V_{t+1}$ be the members of $\pi$ containing $v_{s-1}$ and $v_{t+1}$, respectively. Note that $A$ cannot be a $pc$-partner of any set in $\pi \setminus \lbrace V_{s-1} ,V_{t+1}\rbrace$. Hence, $A$ admits at most two $pc$-partners.

\textbf{Subcase 1.2.} $\lvert A\rvert  \equiv 0 \pmod{2}$. If $\pi^\prime$  has a component with odd order, then as discussed in the previous subcase, $A$ admits at most two $pc$-partner. Otherwise, $G[A]$ has a perfect matching, and so $A$ is not a dominating set in $G$. It follows that there exists an index $i$ such that $\lbrace v_{i-1} ,v_i , v_{i+1}\rbrace \cap A =\emptyset$. Let the sets $V_{i-1}$, $V_i$ and $V_{i+1}$ be the members of $\pi$ containing $v_{i-1}$, $v_i$ and $v_{i+1}$, respectively. Note that $A$ has no $pc$-partner in $\pi \setminus \lbrace V_{i-1},V_i,V_{i+1}\rbrace$. Now if  $V_i \notin \lbrace V_{i-1} ,V_{i+1}\rbrace$, then $A$ cannot form a paired coalition with $V_i$, implying that $A$ admits at most two $pc$-partners. Otherwise, it follows again that $A$ admits at most two $pc$-partners.

By symmetry, $B$ admits at most two $pc$-partners as well. Therefore,  $\lvert \pi \rvert \leq 4$. It remains to  show that $\lvert \pi \rvert \neq 4$. Suppose, to the contrary, that $\lvert \pi \rvert =4$. Let $C$ and $D$ be the remaining members of $\pi$ such that $C$ is a $pc$-partner of $A$, and $D$ is a $pc$-partner of $B$. It follows that $\lvert A\rvert +\lvert C\rvert =\lvert B\rvert +\lvert D\rvert =\frac{n}{2}$, implying that $n \equiv 0 \pmod{4}$, a contradiction.

\textbf{Case 2.} $n \equiv 0 \pmod{4}$. Consider the vertex partition $\pi =\lbrace A,B,C,D\rbrace$ of $G$ such that
$A = \bigcup_{i=0} ^{\frac{n}{4} -1} \lbrace v_{4i}\rbrace$, $B = \bigcup_{i=0} ^{\frac{n}{4} -1} \lbrace v_{4i+1}$$\rbrace$, $C = \bigcup_{i=0} ^{\frac{n}{4} -1} \lbrace v_{4i+2}\rbrace$ and $D = \bigcup_{i=0} ^{\frac{n}{4} -1} \lbrace v_{4i+3}\rbrace$. One can observe that $\pi$ is a $pc$-partition of $G$, where $A$ forms a paired coalition with $B$, and $C$ forms a paired coalition with $D$. Thus,  $\lvert \pi \rvert \geq 4$. Now we show that $\lvert \pi \rvert \leq 4$. Let the sets $A$ and $B$ be $pc$-partners in $\pi$. We  consider two subcases.

\textbf{Subcase 2.1.} $\lvert A\rvert +\lvert B\rvert > \frac{n}{2}$. It follows that for any  pair $(X,Y)$ of $pc$-partners, $A \in \lbrace X,Y\rbrace$ or $B \in \lbrace X,Y\rbrace$. Now similar to the previous case, we can show that $\lvert \pi \rvert \leq 4$, and so $\lvert \pi \rvert =4$.

\textbf{Subcase 2.2.} $\lvert A\rvert +\lvert B\rvert = \frac{n}{2}$. Suppose, to the contrary, that $\lvert \pi \rvert \geq 5$. Observe that for any  pair $(X,Y)$ of $pc$-partners, $A \in \lbrace X,Y\rbrace$ or $B \in \lbrace X,Y\rbrace$. On the other hand, similar to the previous case, we can show that each set in $\lbrace A,B\rbrace$ admits at most two $pc$-partners, a contradiction. Hence, $\lvert \pi \rvert \leq 4$, and so $\lvert \pi \rvert =4$.
\end{proof}
\begin{corollary} \label{delta-c} 
	If $C$ is a cycle, and $\pi$ is a $pc$-partition of $C$, then $\Delta (PCG(C,\pi))\leq 2$.
\end{corollary}
\begin{theorem}
	A graph $G$ is a $PCC$-graph if and only if $G \in \lbrace P_3,K_3,K_2 \cup K_2, P_4,C_4\rbrace$.
\end{theorem}
\begin{proof}
Consider the cycles $C_3 =(v_1,v_2,v_3)$, $C_4=(v_1,v_2,v_3,v_4)$, $C_5=(v_1, \dots ,v_5)$ and $C_8 =(v_1, \dots ,v_8)$, and  the vertex partitions $A=\lbrace \lbrace v_1\rbrace ,\lbrace v_2\rbrace ,\lbrace v_3\rbrace \rbrace$, $B=\lbrace \lbrace v_1\rbrace ,\lbrace v_2\rbrace ,\lbrace v_3\rbrace ,\lbrace v_4\rbrace \rbrace$, $C=\lbrace \lbrace v_1\rbrace ,\lbrace v_2\rbrace ,\lbrace v_3,v_4,v_5\rbrace \rbrace$, $D=\lbrace \lbrace v_1,v_2,v_5\rbrace ,\lbrace v_3,v_7,v_8\rbrace ,\lbrace v_6\rbrace ,\lbrace v_4\rbrace \rbrace$, and \\ $E=\lbrace \lbrace v_1,v_2,v_6\rbrace ,\lbrace v_3,v_7,v_8\rbrace ,\lbrace v_5\rbrace ,\lbrace v_4\rbrace \rbrace$.

 One can observe that $PCG(C_3, A) \simeq K_3$, $PCG(C_5, C) \simeq P_3$, $PCG(C_4,B) \simeq C_4$, $PCG(C_8, D) \simeq K_2 \cup K_2$, and $PCG(C_8, E) \simeq P_4$. Conversely, let $G$ be a $PCC$-graph of order $n$. By Theorem \ref{pc-cycle}, we have $3\leq n\leq 4$, and by Corollary \ref{delta-c}, $\Delta (G) \leq 2$. Further, note that $G$ has no isolated vertex. Hence, $G \in \lbrace P_3,K_3,K_2 \cup K_2, P_4,C_4 \rbrace$, which completes the proof.
\end{proof}
Note that  not all graphs have a $pc$-partition. In the next section, we will construct a family of trees for which no $pc$-partition exists.
	\section{Trees}
The following two corollaries are immediate results of Lemma \ref{lem.sup2} and Corollary \ref{coro-all}, respectively.
\begin{corollary} \label{coro-sup}
	Let $T$ be  tree with $PC(T)\geq 3$, and let $\pi$ be  a $PC(T)$-partition. Then $\pi$ has a member containing all support vertices of $T$.
\end{corollary}
\begin{corollary} \label{coro-t}
	Let $T$ be a tree with $PC(T)\geq 3$, and let $\pi$ be a $PC(T)$-partition. Further, let $V_s$ be the member of $\pi$ containing support vertices of $T$. Then for any pair $(A,B)$ of $pc$-partners in $T$, $V_s \in \lbrace A,B\rbrace$.
\end{corollary}
	Note that a tree with paired coalition number two  has a perfect matching and a tree containing at least one strong support vertex has no perfect matching. Therefore, we have the following result.
	\begin{observation} \label{obs-two} 
	Let $T$ be a tree with at least one strong support vertex. If $T$ admits a $pc$-partition, then $PC(T) \geq 3$. 
\end{observation}
\begin{lemma} \label{lem-sup3}
	Let $T$ be a tree containing a strong support vertex $v$, and let $U=\lbrace v_1,v_2, \dots ,v_k\rbrace$ be the set of leaves adjacent to $v$. Further, let $\pi$ be a $PC(T)$-partition, and let $V_s$ be the member of $\pi$ containing support vertices of $T$. Then  $V_s \cap U=\emptyset$.
\end{lemma}
\begin{proof}
	Suppose, to the contrary, that $V_s \cap U\neq \emptyset$. Let $A=V_s \cap U$. By Observation \ref{obs-two},  $PC(T)\geq 3$, and so by Corollary \ref{coro-t}, $V_s$ forms a paired coalition with all other sets. If $\lvert A\rvert \geq 2$, then for any set $V_i \in \pi \setminus \lbrace V_s\rbrace$, $T[ V_s \cup V_i]$ does not have a perfect matching, and so $V_s$ has no $pc$-partner, a contradiction. Hence, we may assume that $\lvert A\rvert=1$. Consider an arbitrary leaf $v_r \in U \setminus A$, and let $V_r$ be the member of $\pi$ containing $v_r$. Now we observe that $T[ V_s \cup V_r] $ has no perfect matching, implying that the sets $V_r$ and $V_s$ are not $pc$-partners, a contradiction. Hence, $V_s \cap U =\emptyset$.
\end{proof}

The following lower bound can be obtained for trees having a $pc$-partition and at least one strong support vertex.

\begin{theorem} \label{tree-bound}
	Let $T$ be a tree containing at least one strong support vertex, and let $\lbrace v_1,v_2,\dots ,v_t\rbrace$ be the set of  strong support vertices of $T$. For each $1\le i\leq t$, let $l_i$ be the number of leaves adjacent to $v_i$. If there exists a $pc$-partition for $T$,  then 
 $$ PC(T) \geq \max \{l_i:i=1,\dots ,t\}+1. $$	
\end{theorem}
\begin{proof}
	Let $l =\max \{l_i : i=1, \dots ,t\}$, and let $v$ be a strong support vertex with $l$ leaves adjacent to it, where $U=\lbrace u_1,u_2,\dots ,u_{l} \rbrace$ is the set of leaves adjacent to $v$. Further, Let $\pi$ be a $PC(T)$-partition. If there exists a set $V_i$ in $\pi$ containing at least two vertices of $U$, then $V_i$ would not form a paired coalition with any other sets. Hence, each set in $\pi$  contains at most one vertex from $U$, implying that $\lvert \pi \rvert \geq l$. Furthermore, Observation \ref{obs-two} implies that $PC(T) \geq 3$. Thus, by Corollary \ref{coro-sup}, there exists a set $V_s \in \pi$  that  contains all support vertices of $T$. Now Lemma \ref{lem-sup3} implies that  $V_s \cap U =\emptyset$, and so $\lvert \pi \rvert \geq l+1$.  This completes the proof.
\end{proof}
Next we present a family of trees for which no $pc$-partition exists.
\begin{proposition} 
	Let $T$ be a tree containing a perfect matching, and let $T^\prime$ be the tree obtained from $T$ by adding at least two  leaf neighbors to each vertex of $T$. Then $T'$ has no $pc$-partition. 
\end{proposition}
\begin{proof}
		Suppose, to the contrary, that $T^\prime$ has a $pc$-partition. Let $\pi$ be a $PC(T^\prime)$-partition. It follows from Theorem \ref{tree-bound} that  $PC(T^\prime) \geq  3$. Let $S$ be the set of support vertices of $T^\prime$, and let $V_s$ be the member of $\pi$ containing  $S$. By the construction of $T^\prime$, (Figure \ref{pic1}) we have $S=V(T)$. Furthermore, each vertex in $S$ is a strong support vertex and each vertex in $V(T^\prime) \setminus S$ is a leaf. Now Lemma \ref{lem-sup3} implies that $V_s =S$, and so $V_s =V(T)$. Hence, $T[V_s]$ has a perfect matching, and so  $V_s$ is a paired dominating set in $T^\prime$, a contradiction. Therefore, the result follows.  	
\end{proof}
Figure \ref{pic1} illustrates two examples of such a construction, where the original tree $T$ is colored blue. 
		\begin{figure}[!htbp]
		\centering
		\begin{subfigure}{0.4\textwidth}
		\begin{tikzpicture}[scale=.3, transform shape]
			\node [draw, shape=circle,fill=black,color=blue] (v1) at  (0,0) {};
			\node [draw, shape=circle,fill=black,color=blue] (v2) at  (4,0) {};
			\node [draw, shape=circle,fill=black,color=blue] (v3) at  (8,0) {};
			\node [draw, shape=circle,fill=black,color=blue] (v4) at  (12,0) {};
			\node [draw, shape=circle,fill=black,color=red] (v5) at  (-1,4) {};
			\node [draw, shape=circle,fill=black,color=red] (v0) at  (0,4) {};
			\node [draw, shape=circle,fill=black,color=red] (v6) at  (1,4) {};
			\node [draw, shape=circle,fill=black,color=red] (v7) at  (3,4) {};
			\node [draw, shape=circle,fill=black,color=red] (v8) at  (5,4) {};
			\node [draw, shape=circle,fill=black,color=red] (v9) at  (7,4) {};
			\node [draw, shape=circle,fill=black,color=red] (v10) at  (9,4) {};
			\node [draw, shape=circle,fill=black,color=red] (v11) at  (11,4) {};
			\node [draw, shape=circle,fill=black,color=red] (v13) at  (12,4) {};
			\node [draw, shape=circle,fill=black,color=red] (v12) at  (13,4) {};
			\node [draw, shape=circle,fill=black,color=red] (v14) at  (14,4) {};
			
			\draw[color=red] (v1)--(v5);
		
			\draw[color=red](v1)--(v0);
			\draw[color=red](v1)--(v6);
			\draw[color=red](v2)--(v7);
			\draw[color=red](v2)--(v8);
			\draw[color=red](v3)--(v9);
			\draw[color=red](v3)--(v10);
			\draw[color=red](v4)--(v11);
			\draw[color=red](v4)--(v12);
			\draw[color=red](v4)--(v13);
			\draw[color=red](v4)--(v14);
			\draw[color=blue](v1)--(v2);
			\draw[color=blue](v2)--(v3);
			\draw[color=blue](v3)--(v4);			
			
		\end{tikzpicture}
	
	\end{subfigure}
\begin{subfigure}{0.4\textwidth}
	\begin{tikzpicture}[scale=.3, transform shape]
		\node [draw, shape=circle,fill=black,color=blue] (v1) at  (0,0) {};
		\node [draw, shape=circle,fill=black,color=blue] (v2) at  (4,0) {};
	\node [draw, shape=circle,fill=black,color=blue] (v3) at  (8,0) {};
	\node [draw, shape=circle,fill=black,color=blue] (v4) at  (12,0) {};
	\node [draw, shape=circle,fill=black,color=blue] (v5) at  (16,0) {};

	\node [draw, shape=circle,fill=black,color=blue] (v6) at  (0,4) {};
	\node [draw, shape=circle,fill=black,color=blue] (v7) at  (4,4) {};
	\node [draw, shape=circle,fill=black,color=blue] (v8) at  (8,4) {};
	\node [draw, shape=circle,fill=black,color=blue] (v9) at  (12,4) {};
	\node [draw, shape=circle,fill=black,color=blue] (v10) at  (16,4) {};
		
	\node [draw, shape=circle,fill=black,color=red] (v11) at  (-1,-4) {};
	\node [draw, shape=circle,fill=black,color=red] (v12) at  (1,-4) {};
	\node [draw, shape=circle,fill=black,color=red] (v13) at  (3,-4) {};
	\node [draw, shape=circle,fill=black,color=red] (v31) at  (4,-4) {};
	
	\node [draw, shape=circle,fill=black,color=red] (v14) at  (5,-4) {};
\node [draw, shape=circle,fill=black,color=red] (v15) at  (7,-4) {};
\node [draw, shape=circle,fill=black,color=red] (v16) at  (9,-4) {};	
	\node [draw, shape=circle,fill=black,color=red] (v17) at  (11,-4) {};
	\node [draw, shape=circle,fill=black,color=red] (v18) at  (13,-4) {};
	\node [draw, shape=circle,fill=black,color=red] (v19) at  (15,-4) {};
	\node [draw, shape=circle,fill=black,color=red] (v32) at  (16,-4) {};
	
	\node [draw, shape=circle,fill=black,color=red] (v20) at  (17,-4) {};
			\node [draw, shape=circle,fill=black,color=red] (v21) at  (-1,8) {};
		\node [draw, shape=circle,fill=black,color=red] (v22) at  (1,8) {};
			\node [draw, shape=circle,fill=black,color=red] (v33) at  (0,8) {};
				\node [draw, shape=circle,fill=black,color=red] (v34) at  (-2,8) {};
		\node [draw, shape=circle,fill=black,color=red] (v23) at  (3,8) {};
		\node [draw, shape=circle,fill=black,color=red] (v24) at  (5,8) {};
		\node [draw, shape=circle,fill=black,color=red] (v25) at  (7,8) {};
		\node [draw, shape=circle,fill=black,color=red] (v26) at  (9,8) {};
				\node [draw, shape=circle,fill=black,color=red] (v35) at  (8,8) {};		
		\node [draw, shape=circle,fill=black,color=red] (v27) at  (11,8) {};
		\node [draw, shape=circle,fill=black,color=red] (v28) at  (13,8) {};
		\node [draw, shape=circle,fill=black,color=red] (v36) at  (12,8) {};
		\node [draw, shape=circle,fill=black,color=red] (v29) at  (15,8) {};
		\node [draw, shape=circle,fill=black,color=red] (v30) at  (17,8) {};
		\draw[color=blue] (v1)--(v2);
		\draw[color=blue](v2)--(v3);
		\draw[color=blue](v3)--(v4);
		\draw[color=blue](v4)--(v5);
		\draw[color=blue](v6)--(v7);
		\draw[color=blue](v7)--(v8);
		\draw[color=blue](v8)--(v9);
		\draw[color=blue](v9)--(v10);
		\draw[color=blue](v3)--(v8);
	
		\draw[color=red] (v1)--(v11);
\draw[color=red](v1)--(v12);
\draw[color=red](v2)--(v13);
\draw[color=red](v2)--(v14);
\draw[color=red](v2)--(v31);

\draw[color=red](v3)--(v15);
\draw[color=red](v3)--(v16);
\draw[color=red](v4)--(v17);
\draw[color=red](v4)--(v18);
\draw[color=red](v5)--(v19);	
\draw[color=red](v5)--(v20);	
\draw[color=red](v5)--(v32);
		\draw[color=red] (v6)--(v21);
\draw[color=red](v6)--(v22);
\draw[color=red](v6)--(v33);
\draw[color=red](v6)--(v34);
\draw[color=red](v7)--(v23);
\draw[color=red](v7)--(v24);
\draw[color=red](v8)--(v25);
\draw[color=red](v8)--(v26);
\draw[color=red](v8)--(v35);

\draw[color=red](v9)--(v27);
\draw[color=red](v9)--(v28);
\draw[color=red](v9)--(v36);
\draw[color=red](v10)--(v29);	
\draw[color=red](v10)--(v30);	
	\end{tikzpicture}

\end{subfigure}
	\caption{Trees with no $pc$-partition}\label{pic1}
\end{figure}
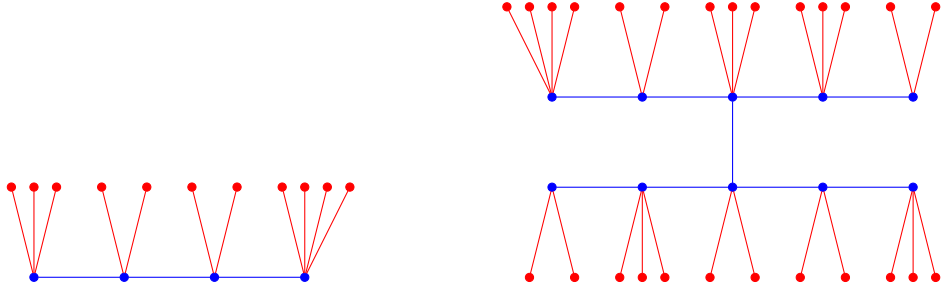

The next theorem characterizes paired coalition graphs of all trees.
\begin{theorem}
	A graph $G \neq K_1$ of order $n$ is   a $PCT$-graph if and only if $G \simeq K_{1,n-1}$.
\end{theorem}
\begin{proof}
	It follows from Corollary \ref{pcg-1} that if $G$ is $PCT$-graph, then $G \simeq K_{1,n-1}$. To prove the converse, for each $n\geq 2$, we construct a tree $T$ and a $pc$-partition $\pi$ of $T$ such that $PCG(T,\pi)\simeq K_{1,n-1}$.  Let $T$ be a tree of order $n+2$ that is  obtained from   subdividing the edge  connecting  support vertices of a double star $S_{p,q}$, for some $p,q\geq 1$, where $\lbrace v_1,v_2,\dots v_p, u_1,u_2,\dots ,u_q\rbrace$ is the set of leaves of $S_{p,q}$, and $\lbrace x,y,z\rbrace$ is the set of non-leaf vertices of $S_{p,q}$. Then it is easily seen that the partition  $\pi =\lbrace \lbrace v_1\rbrace ,\lbrace v_2\rbrace,\dots ,\lbrace v_p\rbrace ,\lbrace u_1\rbrace ,\lbrace u_2\rbrace ,\dots ,\lbrace u_q\rbrace ,\lbrace x,y,z\rbrace \rbrace$  is a $pc$-partition for $T$, where $\lbrace x,y,z\rbrace$ forms a paired coalition with all other sets in $\pi$, and no two sets in $\pi \setminus \lbrace \lbrace x,y,z \rbrace \rbrace$ are $pc$-partners. Hence,  $PCG(T,\pi)\simeq K_{1,n-1}$.
\end{proof}
\subsection{Trees with large paired coalition number} 

A vertex partitioning of a graph $G$ of order $n$, into $n-1$ sets, yields $n-2$ singleton sets and a set of  cardinality $2$. Such a partition cannot be a $pc$-partition, since for any pair $(A,B)$ of $pc$-partners, $\lvert A\rvert +\lvert B\rvert$ must be an even number. Thus, $G$ has no $pc$-partition of order $n-1$, implying that for any graph $G$, $PC(G)\neq n-1$.
In this subsection, we characterize trees $T$ of order $n$ with $PC(T)\in \lbrace n,n-2\rbrace$.  

As an immediate result of Lemma \ref{l5} we have the following.
\begin{corollary} \label{tree-pc-n}
For a tree $T \neq K_1$ of order $n$, $PC(T)=n$ if and only if $T=K_{1,n-1}$.
\end{corollary}

\begin{theorem}  \label{tree-n-2}
	Let $T$ be a tree of order $n\geq 4$. Then $PC(T)=n-2$ if and only if either $T \simeq S_{1,p}$, for some $p\geq 1$, or  $T$ is obtained from a double star $S_{p,q}$, for some $p\geq 1$ and $q\geq 1$, by subdividing  the edge connecting  its support vertices.
\end{theorem} \label{tree-n-2}
\begin{proof}
	Let $T$ be a tree of order $n$ that is obtained from a double star $S_{p,q}$ by subdividing  the edge connecting its support vertices. Let $u$ and $v$  be the support vertices of $T$, where $N(u) \cap N(v)=\lbrace s\rbrace$. Further, let $u_1,u_2,\dots ,u_p$ be the leaves adjacent to $u$, and let $v_1,v_2,\dots ,v_q$ be the leaves adjacent to $v$. Observe  that the singleton partition of $V(T)$ is not a $pc$-partition of $T$, and that  the partition 
	$\pi = \lbrace \lbrace u,s,v\rbrace, \lbrace u_1\rbrace ,\lbrace u_2\rbrace \dots , \lbrace u_p\rbrace , \lbrace v_1\rbrace ,\lbrace v_2\rbrace ,\dots ,\lbrace v_q \rbrace \rbrace$ is a $pc$-partition of $T$ of order $n-2$, where the set $\lbrace u,s,v\rbrace$ forms a paired coalition with all other sets in $\pi$. Hence,  $PC(T)=n-2$.

	Now assume $T \simeq S_{1,p}$ for some $p\geq 1$. Let $\lbrace u,v\rbrace$ be the set of support vertices of $T$ such that $N(u)=\lbrace v,t\rbrace$. Since $T$ has no full vertex, it follows from Lemma \ref{l5} that $PC(T)\neq n$. Now consider the vertex partition $\pi$ of $T$, where $\lbrace u,v,t\rbrace \in \pi$ and the remaining members of $\pi$ are singleton sets. Observe that $\pi$ is a $pc$-partition of $T$ of order $n-2$, where the set $\lbrace u,v,t\rbrace$ forms a paired coalition with all other sets in $\pi$. Hence,  $PC(T)=n-2$.

	Conversely, assume that $PC(T)=n-2$. Let $\pi$ be a $PC(T)$-partition.  The result can be readily verified when $n \in \lbrace 4,5\rbrace$, so we may assume that $n\geq 6$. 
	Consider two cases.
	
	\textbf{Case 1.} $\pi$ consists of two sets of cardinality $2$ and $n-4$ singleton sets. Let $\lbrace A,B\rbrace \subset \pi$, where $A=\lbrace x,y\rbrace$, $B=\lbrace z,t\rbrace$. Note that $A$ and $B$ must be $pc$-partners. Note also that neither $A$ nor $B$ can form a paired coalition with any singleton set in $\pi$. Let $\lbrace c\rbrace \in \pi$, and $\lbrace d\rbrace \in \pi$   be  $pc$-partners in $\pi$. Since  $T[A \cup B]$ contains a perfect matching, $x$ has a neighbor in $\lbrace y,z,t\rbrace$. Assume, without loss of generality, that $xy \in E(T[A\cup B] )$. Now since $cd$ is a dominating edge, each vertex in $\lbrace x,y \rbrace $ has a neighbor in $\lbrace c,d\rbrace$, which creates cycle. Hence, this case is impossible.
	
	\textbf{Case 2.} $\pi$ consists of a set of cardinality  $3$ and $n-3$ singleton sets. Let $\pi =\lbrace A,V_1,V_2, \dots ,V_{n-3}\rbrace$, where  $A=\lbrace x,y,z\rbrace$ and  $V_i=\lbrace v_i\rbrace$ for each $1\leq i\leq n-3$, and let $V_j$ be a $pc$-partner of $A$. Note that $T[A]$ is not an empty graph, for otherwise, $T[A \cup V_j]$ would not have a perfect matching. Assume, without loss of generality, that $xy \in E(T[A])$. Now we show that no two sets in $\pi \setminus \lbrace A\rbrace$ are $pc$-partners. Suppose that the converse is true. Let $V_k =\lbrace v_k\rbrace$ and $V_l=\lbrace v_l\rbrace$ be $pc$-partners in $\pi$. Since $x,y$ are dominated by $ V_k \cup V_l $, $T[\lbrace x,y,v_k,v_l\rbrace]$ contains a cycle, a contradiction. Hence, $A$ forms a paired coalition with all other sets in $\pi$.
	
	 If $E(T[A])=\lbrace xy\rbrace$, then considering the fact that for each $1\leq i\leq n-3$, $T[\lbrace x,y,z,v_i\rbrace]$ has a perfect matching, we deduce that  $zv_i \in E(T)$, for each $1\leq i\leq n-3$.  This implies that the vertices in $\lbrace v_1,v_2,\dots ,v_{n-3}\rbrace$ are pairwise nonadjacent, and that $\lvert [\lbrace x,y\rbrace , V(T) \setminus  A]\rvert=1$,  as desired. If $E(T[A]) \neq \lbrace xy\rbrace$, then either $xz \in E(T[A])$ or $yz \in E(T[A])$. By symmetry, we assume  $xz \in E(T[A])$. Since $A$ forms a paired coalition with all other sets of $\pi$, each vertex in $\lbrace v_1,v_2,\dots ,v_{n-3}\rbrace$ has a neighbor in $\lbrace x,y,z\rbrace$. So in order to avoid cycles, all vertices in $\lbrace v_1,v_2,\dots ,v_{n-3}\rbrace$ must be pairwise nonadjacent, and each vertex in $\lbrace v_1,v_2,\dots ,v_{n-3}\rbrace$ must have  one neighbor in $\lbrace x,y,z\rbrace$. Further, note that $x$ is not adjacent to any vertex  $ v_k \in \lbrace v_1,v_2,\dots ,v_{n-3}\rbrace$, for otherwise, $T[\lbrace x,y,z,v_k\rbrace ]$ would not have a perfect matching. Thus, each vertex in $\lbrace v_1,v_2,\dots ,v_{n-3}\rbrace$ is adjacent to either $y$ or $z$. If $N_T (\lbrace v_1,v_2,\dots ,v_{n-3} \rbrace) = \lbrace y\rbrace$ or $N_T (\lbrace v_1,v_2,\dots ,v_{n-3} \rbrace) = \lbrace z\rbrace$, then $T$ is a double star $S_{1,p}$ of order $n$, for some $p\geq 1$. Otherwise, $T$ is obtained from a double star $S_{p,q}$, for some $p\geq 1$ and $q\geq 1$, by subdividing  the edge connecting  its support vertices. So the proof is  complete.
\end{proof}

\subsection{Perfect  binary trees} 
A {\em perfect binary tree} is a binary tree in which  all leaves have same distance to the root vertex. Given a perfect binary tree $T$, we call the distance between  root vertex and each leaf, height of $T$. If $T$ has height $h$, then we denote it by $T(h)$. It is clear that $PC(T(1))=3$. Further, by Theorem \ref{tree-n-2},  $PC(T(2))=5$, and by Theorem \ref{tree-bound},  either $PC(T(h))=0$ or $PC(T(h))\geq 3$. In this subsection, we derive an upper bound on $PC(T(h))$, for $h\geq 3$, and determine exact values of $PC(T(3))$ and $PC(T(4))$.

 \begin{lemma}
 	If $h\geq 3$, then $PC(T(h)) \leq 4$.
 \end{lemma}
\begin{proof}
	First we show that $PC(T(h))\leq 5$. Suppose, to the contrary, that $PC(T(h))\geq 6$. Let $\pi$ be a $PC(T(h))$-partition, and let $ V_1 \in \pi$ be the set containing support vertices of $T(h)$.  Consider a subgraph $T^\prime \subset T(h)$, where $T^\prime \simeq T(2)$, and the leaves in $T^\prime$ are also leaves in $T(h)$ (The graph $T^\prime$ is illustrated in Figure 2). Note that $\lbrace v_2,v_3\rbrace \subset V_1$. Since $PC(T(h))\geq 6$, there exists a set $V_i \in \pi \setminus \lbrace V_1\rbrace$, such that $V_i \cap \lbrace v_4,v_5,v_6,v_7\rbrace =\emptyset$. Further, by Lemma \ref{lem-sup3},  $V_1 \cap \lbrace v_4,v_5,v_6,v_7\rbrace =\emptyset$. So $N_{T(h) [V_1 \cup V_i]} (v_2) \cup N_{T(h) [V_1 \cup V_i]} (v_3) \subseteq \lbrace v_1\rbrace$, which means that $T(h) [V_1 \cup V_i]$ has no perfect matching. But,  Corollary \ref{coro-t} implies that $V_i$ must form a paired coalition with $V_1$, a contradiction.
	Now we show that  $PC(T(h))\neq 5$. Suppose, to the contrary, that $PC(T(h))=5$. Let $\pi =\lbrace V_1,V_2,V_3,V_4,V_5\rbrace$ be a $PC(T(h))$-partition, and let $V_1$ be the set containing support vertices of $T(h)$. Consider two cases.
	
	\textbf{Case 1.} $h=3$. We label the vertices of $T(3)$ as shown in Figure 3.
	Note that $\lbrace v_4,v_5,v_6,v_7\rbrace \subset V_1$. As discussed above, for each set $V_i \in \pi \setminus \lbrace V_1\rbrace$, $V_i \cap \lbrace v_8,v_9,v_{10},v_{11}\rbrace \neq \emptyset$, and  $V_i \cap \lbrace v_{12},v_{13},v_{14},v_{15}\rbrace \neq \emptyset$. Therefore, each set in $\pi \setminus \lbrace V_1\rbrace$ contains exactly one vertex from $\lbrace v_8,v_9,v_{10},v_{11}\rbrace$, and one vertex from $\lbrace v_{12},v_{13},v_{14},v_{15}\rbrace$. Considering the fact that the induced subgraphs $T(h) [V_1 \cup V_2]$ and $T(h) [V_1 \cup V_3]$ contain a perfect matching, we deduce that $\lbrace v_2,v_3\rbrace \subset V_1$. Now Considering the fact that  $T(h) [V_1 \cup V_2]$ contains a perfect matching, we deduce that the vertices $v_2$ and $v_3$ are matched to one of their children which implies that $v_1$ is not matched to any other vertices, a contradiction.
	
	\textbf{Case 2.} $h >3$.
	Consider a subgraph $T^\prime \subseteq T(h)$, where $T^\prime \simeq T(4)$, and the leaves in $T^\prime$ are also leaves in $T(h)$ (The graph $T^\prime$ is illustrated in Figure 5). Note that $\lbrace v_8,v_9,v_{10},v_{11},v_{12},v_{13},v_{14},v_{15}\rbrace \subset V_1$. As discussed in the previous case,  $\lbrace v_4,v_5,v_6,v_7\rbrace \subset V_1$, and each vertex in $\lbrace v_4,v_5,v_6,v_7\rbrace$ is matched to one of its children. Hence, both $v_2$ and $v_3$ must be matched to $v_1$ which is a contradiction.
\end{proof}
		\begin{figure}[!htbp]
	\centering
	\begin{tikzpicture}[scale= .15, transform shape]
		\node [draw, shape=circle,fill=black,scale=2] (v1) at  (0,20) {};
		\node [draw, shape=circle,fill=black,scale=2] (v2) at  (-10,16) {};
		\node [draw, shape=circle,fill=black,scale=2] (v3) at  (10,16) {};
		\node [draw, shape=circle,fill=black,scale=2] (v4) at  (-15,12) {};
		\node [draw, shape=circle,fill=black,scale=2] (v5) at  (-5,12) {};
		\node [draw, shape=circle,fill=black,scale=2] (v6) at  (5,12) {};
		\node [draw, shape=circle,fill=black,scale=2] (v7) at  (15,12) {};
		
		\node [scale=6] at (0,21.5) {$v_1$};
		\node [scale=6] at (-10,17.5) {$v_2$};
		\node [scale=6] at (10,17.5) {$v_3$};
		\node [scale=6] at (-15,10.5) {$v_4$};
		\node [scale=6] at (-5,10.5) {$v_5$};
		\node [scale=6] at (5,10.5) {$v_6$};
		\node [scale=6] at (15,10.5) {$v_7$};
		\draw(v1)--(v2);
		\draw(v1)--(v3);
		\draw(v2)--(v4);
		\draw(v2)--(v5);
		\draw(v3)--(v6);
		\draw(v3)--(v7);

	\end{tikzpicture}
	\caption{The tree $T(2)$}\label{tree1}
\end{figure}
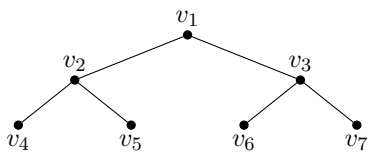
 \begin{figure}[!htbp]
	\centering
	
	\begin{tikzpicture}[scale=.3, transform shape]
		\node [draw, shape=circle,fill=black] (v1) at  (-5,10) {};
		\node [draw, shape=circle,fill=black] (v2) at  (-11,8) {};
		\node [draw, shape=circle,fill=black] (v3) at  (1,8) {};
		\node [draw, shape=circle,fill=black] (v4) at  (-14,6) {};
		\node [draw, shape=circle,fill=black] (v5) at  (-8,6) {};
		\node [draw, shape=circle,fill=black] (v6) at  (-2,6) {};
		\node [draw, shape=circle,fill=black] (v7) at  (4,6) {};
		\node [draw, shape=circle,fill=black] (v8) at  (-15.5,4) {};
		\node [draw, shape=circle,fill=black] (v9) at  (-12.5,4) {};
		\node [draw, shape=circle,fill=black] (v10) at  (-9.5,4) {};
		\node [draw, shape=circle,fill=black] (v11) at  (-6.5,4) {};
		\node [draw, shape=circle,fill=black] (v13) at  (-3.5,4) {};
		\node [draw, shape=circle,fill=black] (v12) at  (-0.5,4) {};
		\node [draw, shape=circle,fill=black] (v14) at  (2.5,4) {};
		\node [draw, shape=circle,fill=black] (v15) at  (5.5,4) {};
		
		\node [scale=3] at (-5,10.7) {$v_1$};
		\node [scale=3] at (-11,8.7) {$v_2$};
		\node [scale=3] at (1,8.7) {$v_3$};
		\node [scale=3] at (-14,6.7) {$v_4$};
		\node [scale=3] at (-8,6.7) {$v_5$};
		\node [scale=3] at (-2,6.7) {$v_6$};
		\node [scale=3] at (4,6.7) {$v_7$};
		\node [scale=3] at (-15.5,3.3) {$v_8$};
		\node [scale=3] at (-12.5,3.3) {$v_9$};
		\node [scale=3] at (-9.5,3.3) {$v_{10}$};
		\node [scale=3] at (-6.5,3.3) {$v_{11}$};
		\node [scale=3] at (-3.5,3.3) {$v_{12}$};
		\node [scale=3] at (-0.5,3.3) {$v_{13}$};		
		\node [scale=3] at (2.5,3.3) {$v_{14}$};
		\node [scale=3] at (5.5,3.3) {$v_{15}$};
		
		\draw(v1)--(v2);
		\draw(v1)--(v3);
		\draw(v2)--(v4);
		\draw(v2)--(v5);
		\draw(v3)--(v6);
		\draw(v3)--(v7);			
		\draw(v1)--(v2);
		\draw(v4)--(v8);
		\draw(v4)--(v9);
		\draw(v5)--(v10);
		\draw(v5)--(v11);
		\draw(v6)--(v12);
		\draw(v6)--(v13);
		\draw(v7)--(v14);
		\draw(v7)--(v15);								
	\end{tikzpicture}
	\caption{The tree $T(3)$}\label{tree2}
\end{figure}
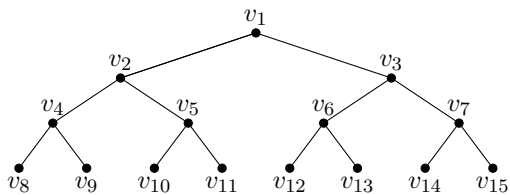
\begin{figure}[!htbp]
	\centering
	
	\begin{tikzpicture}[scale=.3, transform shape]
		\node [draw, shape=circle,fill=black] (v1) at  (-5,10) {};
		\node [draw, shape=circle,fill=black] (v2) at  (-11,8) {};
		\node [draw, shape=circle,fill=black] (v3) at  (1,8) {};
		\node [draw, shape=circle,fill=black] (v4) at  (-14,6) {};
		\node [draw, shape=circle,fill=black] (v5) at  (-8,6) {};
		\node [draw, shape=circle,fill=black] (v6) at  (-2,6) {};
		\node [draw, shape=circle,fill=black] (v7) at  (4,6) {};
		\node [draw, shape=circle,fill=black] (v8) at  (-15.5,4) {};
		\node [draw, shape=circle,fill=black] (v9) at  (-12.5,4) {};
		\node [draw, shape=circle,fill=black] (v10) at  (-9.5,4) {};
		\node [draw, shape=circle,fill=black] (v11) at  (-6.5,4) {};
		\node [draw, shape=circle,fill=black] (v13) at  (-3.5,4) {};
		\node [draw, shape=circle,fill=black] (v12) at  (-0.5,4) {};
		\node [draw, shape=circle,fill=black] (v14) at  (2.5,4) {};
		\node [draw, shape=circle,fill=black] (v15) at  (5.5,4) {};
		
		\node [scale=3] at (-5,10.7) {1};
		\node [scale=3] at (-11,8.7) {2};
		\node [scale=3] at (1,8.7) {3};
		\node [scale=3] at (-14,6.7) {1};
		\node [scale=3] at (-8,6.7) {1};
		\node [scale=3] at (-2,6.7) {1};
		\node [scale=3] at (4,6.7) {1};
		\node [scale=3] at (-15.5,3.3) {2};
		\node [scale=3] at (-12.5,3.3) {3};
		\node [scale=3] at (-9.5,3.3) {2};
		\node [scale=3] at (-6.5,3.3) {3};
		\node [scale=3] at (-3.5,3.3) {2};
		\node [scale=3] at (-0.5,3.3) {3};		
		\node [scale=3] at (2.5,3.3) {2};
		\node [scale=3] at (5.5,3.3) {3};
		
		\draw(v1)--(v2);
		\draw(v1)--(v3);
		\draw(v2)--(v4);
		\draw(v2)--(v5);
		\draw(v3)--(v6);
		\draw(v3)--(v7);			
		\draw(v1)--(v2);
		\draw(v4)--(v8);
		\draw(v4)--(v9);
		\draw(v5)--(v10);
		\draw(v5)--(v11);
		\draw(v6)--(v12);
		\draw(v6)--(v13);
		\draw(v7)--(v14);
		\draw(v7)--(v15);								
	\end{tikzpicture}
	
	\caption{A $pc$-partition of order 3 for $T(3)$}\label{tree3}
\end{figure}
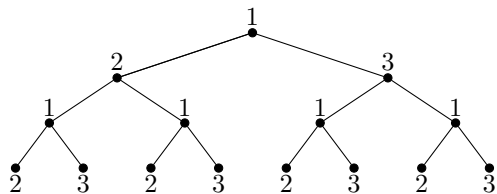
 \begin{proposition} \label{prop-t3}
 	$PC(T(3))=3$.
 \end{proposition}
\begin{proof}
	Figure \ref{tree3} illustrates a $pc$-partition of $T(3)$ of order $3$. So, it remains to show that $PC(T(3))\neq 4$. Suppose, to the contrary, that $PC(T(3))=4$. Let $\pi =\lbrace V_1,V_2,V_3,V_4\rbrace$ be a $PC(T(3))$-partition, and let $V_1$ be the set containing support vertices of $T(3)$.  We label the vertices of $T(3)$ as shown in Figure \ref{tree2}. By Lemma \ref{lem-sup3},  $V_1 \cap \lbrace v_8,v_9,v_{10},v_{11}\rbrace =\emptyset$ and $V_1 \cap \lbrace v_{12},v_{13},v_{14},v_{15}\rbrace =\emptyset$. Note that
	as discussed above, for each set $V_i \in \pi \setminus \lbrace V_1\rbrace$, $V_i \cap \lbrace v_8,v_9,v_{10},v_{11}\rbrace \neq \emptyset$, and  $V_i \cap \lbrace v_{12},v_{13},v_{14},v_{15}\rbrace \neq \emptyset$. Note also that any two leaves with the same parent must be in different sets of $\pi$. Let $A_2$, $A_3$ and $A_4$ denote the number of leaves in $V_2$,$V_3$ and $V_4$, respectively. Therefore
	\[
	\begin{cases}
		A_2 + A_3 +A_4=8 \\
		A_i \geq 2,\  i=2,3,4.
	\end{cases}
	\]
	By symmetry, it suffices to examine the following  two cases.
	
	\textbf{Case 1.} $A_2 =2$, $A_3=2$ and $A_4=4$. It follows that every support vertex of $T(h)$ has a child that is in $V_4$. Therefore, considering a perfect matching in $T(3) [V_1 \cup V_4]$, each support vertex must be matched to one of its children. Hence, both $v_2$ and $v_3$ must be matched to $v_1$ which is a contradiction.
	
	\textbf{Case 2.} $A_2 =2$, $A_3=3$ and $A_4=3$. Considering a perfect matching in $T(3) [V_1 \cup V_2]$, we observe that the vertices $v_2$ and $v_3$ must be matched to one of their children, which implies that $v_1$ is not matched to any other vertices, a contradiction.
\end{proof}

 \begin{proposition}
	$PC(T(4))=0$.
\end{proposition}
		\begin{figure}[!htbp]
	\centering
	\begin{tikzpicture}[scale=.3, transform shape]
		\node [draw, shape=circle,fill=black] (v1) at  (0,20) {};
		\node [draw, shape=circle,fill=black] (v2) at  (-12,16) {};
		\node [draw, shape=circle,fill=black] (v3) at  (12,16) {};
		\node [draw, shape=circle,fill=black] (v4) at  (-18,12) {};
		\node [draw, shape=circle,fill=black] (v5) at  (-6,12) {};
		\node [draw, shape=circle,fill=black] (v6) at  (6,12) {};
		\node [draw, shape=circle,fill=black] (v7) at  (18,12) {};
		\node [draw, shape=circle,fill=black] (v8) at  (-21,8) {};
		\node [draw, shape=circle,fill=black] (v9) at  (-15,8) {};
		\node [draw, shape=circle,fill=black] (v10) at  (-9,8) {};
		\node [draw, shape=circle,fill=black] (v11) at  (-3,8) {};
		\node [draw, shape=circle,fill=black] (v12) at  (3,8) {};
		\node [draw, shape=circle,fill=black] (v13) at  (9,8) {};
		\node [draw, shape=circle,fill=black] (v14) at  (15,8) {};
		\node [draw, shape=circle,fill=black] (v15) at  (21,8) {};
		
		\node [draw, shape=circle,fill=black] (v16) at  (-22.5,4) {};
		\node [draw, shape=circle,fill=black] (v17) at  (-19.5,4) {};
		\node [draw, shape=circle,fill=black] (v18) at  (-16.5,4) {};
		\node [draw, shape=circle,fill=black] (v19) at  (-13.5,4) {};
		\node [draw, shape=circle,fill=black] (v20) at  (-10.5,4) {};
		\node [draw, shape=circle,fill=black] (v21) at  (-7.5,4) {};
		\node [draw, shape=circle,fill=black] (v22) at  (-4.5,4) {};
		\node [draw, shape=circle,fill=black] (v23) at  (-1.5,4) {};
		\node [draw, shape=circle,fill=black] (v24) at  (1.5,4) {};
		\node [draw, shape=circle,fill=black] (v25) at  (4.5,4) {};
		\node [draw, shape=circle,fill=black] (v26) at  (7.5,4) {};
		\node [draw, shape=circle,fill=black] (v27) at  (10.5,4) {};
		\node [draw, shape=circle,fill=black] (v28) at  (13.5,4) {};
		\node [draw, shape=circle,fill=black] (v29) at  (16.5,4) {};
		\node [draw, shape=circle,fill=black] (v30) at  (19.5,4) {};
		\node [draw, shape=circle,fill=black] (v31) at  (22.5,4) {};

		\node [scale=3] at (0,20.7) {$v_1$};
		\node [scale=3] at (-12,16.7) {$v_2$};
		\node [scale=3] at (12,16.7) {$v_3$};
		\node [scale=3] at (-18,12.7) {$v_4$};
		\node [scale=3] at (-6,12.7) {$v_5$};
		\node [scale=3] at (6,12.7) {$v_6$};
		\node [scale=3] at (18,12.7) {$v_7$};
		\node [scale=3] at (-21.5,8.7) {$v_8$};
		\node [scale=3] at (-14.5,8.7) {$v_9$};
		\node [scale=3] at (-9.5,8.7) {$v_{10}$};
		\node [scale=3] at (-2.5,8.7) {$v_{11}$};
		\node [scale=3] at (2.5,8.7) {$v_{12}$};
		\node [scale=3] at (9.5,8.7) {$v_{13}$};		
		\node [scale=3] at (14.5,8.7) {$v_{14}$};
		\node [scale=3] at (21.5,8.7) {$v_{15}$};
		\node [scale=3] at (-22.5,3.3) {$v_{16}$};
		\node [scale=3] at (-19.5,3.3) {$v_{17}$};
		\node [scale=3] at (-16.5,3.3) {$v_{18}$};
		\node [scale=3] at (-13.5,3.3) {$v_{19}$};
		\node [scale=3] at (-10.5,3.3) {$v_{20}$};
		\node [scale=3] at (-7.5,3.3) {$v_{21}$};
		\node [scale=3] at (-4.5,3.3) {$v_{22}$};
		\node [scale=3] at (-1.5,3.3) {$v_{23}$};
		\node [scale=3] at (1.5,3.3) {$v_{24}$};
		\node [scale=3] at (4.5,3.3) {$v_{25}$};
		\node [scale=3] at (7.5,3.3) {$v_{26}$};
		\node [scale=3] at (10.5,3.3) {$v_{27}$};
		\node [scale=3] at (13.5,3.3) {$v_{28}$};
		\node [scale=3] at (16.5,3.3) {$v_{29}$};		
		\node [scale=3] at (19.5,3.3) {$v_{30}$};
		\node [scale=3] at (22.5,3.3) {$v_{31}$};
		\draw(v1)--(v2);
		\draw(v1)--(v3);
		\draw(v2)--(v4);
		\draw(v2)--(v5);
		\draw(v3)--(v6);
		\draw(v3)--(v7);			
		\draw(v1)--(v2);
		\draw(v4)--(v8);
		\draw(v4)--(v9);
		\draw(v5)--(v10);
		\draw(v5)--(v11);
		\draw(v6)--(v12);
		\draw(v6)--(v13);
		\draw(v7)--(v14);
		\draw(v7)--(v15);								
		
		\draw(v1)--(v2);
		\draw(v8)--(v16);
		\draw(v8)--(v17);
		\draw(v9)--(v18);
		\draw(v9)--(v19);
		\draw(v10)--(v20);			
		\draw(v10)--(v21);
		\draw(v11)--(v22);
		\draw(v11)--(v23);
		\draw(v12)--(v24);
		\draw(v12)--(v25);
		\draw(v13)--(v26);
		\draw(v13)--(v27);
		\draw(v14)--(v28);
		\draw(v14)--(v29);		
		\draw(v15)--(v30);
		\draw(v15)--(v31);							
	\end{tikzpicture}
	\caption{The tree $T(4)$}\label{tree4}
\end{figure}
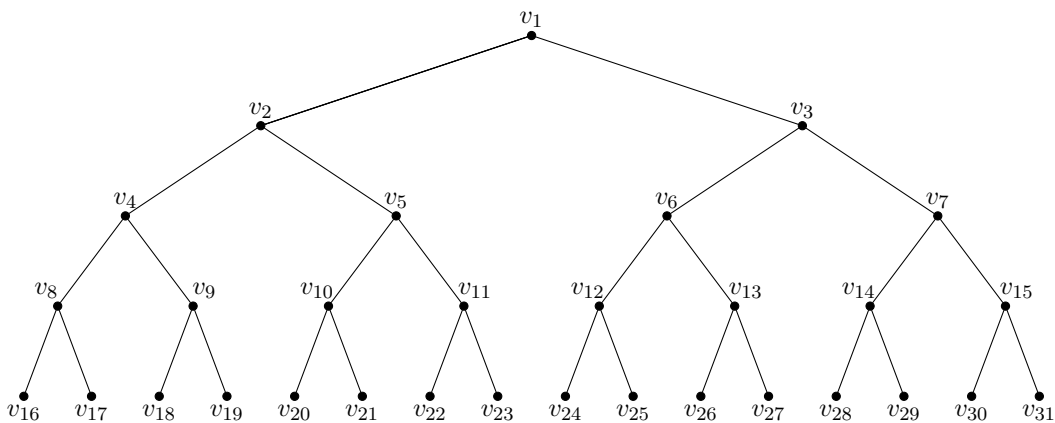
\begin{proof}
	First we show that $PC(T(4))\neq 4$. Suppose, to the contrary, that $PC(T(4))=4$. Let $\pi =\lbrace V_1,V_2,V_3,V_4\rbrace$ be a $PC(T(4))$-partition, and let $V_1$ be the set containing support vertices of $T(4)$. We label the vertices of $T(4)$ as shown in Figure \ref{tree4}.  We denote the two connected components in $T(4) - v_1$ by $T_1$ and $T_2$. Further, Let $A_2$, $A_3$ and $A_4$ denote the number of leaves of $T_1$ in $V_2$,$V_3$ and $V_4$, respectively, and let $B_2$, $B_3$ and $B_4$ denote the number of leaves of $T_2$  in $V_2$,$V_3$ and $V_4$, respectively. As discussed above, we have
	\[
	\begin{cases}
		A_2 + A_3 +A_4=8 \\
		A_i \geq 2, \ i=2,3,4.
	\end{cases}
	\]
	\[
	\begin{cases}
		B_2 + B_3 +B_4=8 \\
		B_i \geq 2, \ i=2,3,4.
	\end{cases}
	\]	
	If there exists a member $k \in \lbrace A_2,A_3,A_4,B_2,B_3,B_4\rbrace$ such that $k=4$, then as discussed in Proposition \ref{prop-t3}, we derive a contradiction. Hence, there exists an index $i$ such that $A_i=B_i=3$. Now considering a perfect matching in $T(4)[V_1 \cup V_i]$, we observe that $v_1$ is not matched to any other vertices, a contradiction.
	
	 Now we show that $PC(T(4)) \neq 3$. Suppose, to the contrary, that $PC(T(4))=3$. Let $\pi =\lbrace V_1,V_2,V_3\rbrace$ be a $PC(T(4))$-partition, and let $V_1$ be the set containing support vertices of $T(4)$. Note that the leaves attached to each support vertex must be in different sets of $\pi$. Thus, one of them must be in $V_2$ and the other one must be in $V_3$. We claim that $v_4 \notin V_1$. Suppose, to the contrary, that $v_4 \in V_1$. Considering a perfect matching in $T[V_1 \cup V_2]$ and a perfect matching in $T[V_1 \cup V_3]$, we deduce that $v_2 \in V_1$. If $v_5 \in V_1$, then considering a perfect matching in $T[V_1 \cup V_2]$, we deduce that both $v_4$ and $v_5$ must be matched to $v_2$, which is impossible. Hence, $v_5 \notin V_1$. Now  if $v_5 \in V_2$, then again considering a perfect matching in $T[V_1 \cup V_2]$, we deduce that both $v_4$ and $v_5$ must be matched to $v_2$, which is impossible. Hence, $v_5 \notin V_2$.
	 
	  A similar argument shows that $v_5 \notin V_3$, a contradiction. Hence, $v_4 \notin V_1$. Similarly, we can show that $v_5 \notin V_1$. Note that $v_4$ and $v_5$ cannot both be in $V_2$. For otherwise, considering a perfect matching in $T[V_1 \cup V_2]$, both $v_4$ and $v_5$ must be matched to $v_2$, which is impossible. A similar argument shows that  $v_4$ and $v_5$ cannot both be in $V_3$. Hence, either $v_4 \in V_2$ and $v_5 \in V_3$, or $v_4 \in V_3$ and $v_5 \in V_2$.
	  
	   Similarly, we can show that either $v_6 \in V_2$ and $v_7 \in V_3$, or $v_6 \in V_3$ and $v_7\in V_2$. Now considering a perfect matching in $T[V_1 \cup V_2]$ and a perfect matching in $T[V_1 \cup V_3]$, we deduce that $v_2 \in V_1$ and $v_3 \in V_1$. Further, considering a perfect matching in $T[V_1 \cup V_2]$ or a perfect matching in $T[V_1 \cup V_3]$, we deduce that both $v_2$ and $v_3$ must be matched to one of their children which implies that $v_1$ is not matched to any other vertex, a contradiction. Hence, $PC(T(4))\neq 3$, and so $PC(T(4))=0$.
\end{proof}
We end this subsection  with the following conjecture:
\begin{con} 
	If $T(h)$ is a perfect binary tree with height $h$, then
	$$
	PC(T(h))= \begin{cases}
		5 & if \  h=2 \\
		3 & if \  \text{h is odd}  \\
		0 &   otherwise.
	\end{cases}
	$$
\end{con}
Note that the pattern we used   to obtain a $pc$-partition of order $3$ for $T(3)$, can be extended and  applied to produce a $pc$-partition of order $3$ for other binary trees with odd height. Also, the arguments we presented to show  $PC(T(4)) \neq 3$, can be extended and applied to prove that $PC(T(h)) \neq 3$, where $h$ is even and $h\geq 6$. So, to prove the conjecture, we need to show that $PC(T(h))\neq 4$, for $h\geq 5$.

\section{Graphs with large paired coalition number}
This section is divided into two subsections. In the first subsection, we characterize all triangle-free graphs of order $n$  with paired coalition number $n$, and in the second subsection, we characterize all unicyclic graphs  of order $n$ with paired coalition number $n-2$.
	\subsection{Triangle-free graphs $G$ with  $PC(G)=n$}
	
\begin{lemma} \label{n-girth} 
	If $G$ is a graph of order $n$ with $PC(G)=n$, then $g(G)\leq 4$.
\end{lemma}
\begin{proof}
	Suppose, to the contrary, that $g(G)\geq 5$. Let  $C$ be an induced cycle in $G$ with order $g(G)$,  and let $\pi$ be the $PC(G)$-partition. Consider an arbitrary set $\lbrace v\rbrace \in \pi$ such that $\lbrace v\rbrace \subset V(C)$. Since $\lvert V(C)\rvert \geq 5$, it follows that $\gamma_{pr} (C)> 3$. Thus, for each $\lbrace u\rbrace \in \pi$, if $\lbrace u\rbrace$ is a $pc$-partner of $\lbrace v\rbrace$, then $\lbrace u\rbrace \subset V(G) \setminus V(C)$, implying that $\lbrace u\rbrace$ must dominate all  vertices in $V(C) \setminus N_C [v]$, which creates cycles of order less than five, a contradiction.
\end{proof}	
The following theorem characterizes  graphs $G$ with $g(G)=4$ and $PC(G)=n$. Note that a complete bipartite graph has girth $4$, if and only if each partite set has cardinality at least $2$.
	\begin{theorem} \label{the-girth4}
		Let $G$ be a graph of order $n$ with $g(G)=4$. Then $PC(G)=n$ if and only if $G$ is a complete bipartite graph.
	\end{theorem}
	\begin{proof} 
	If $G$ is a complete bipartite graph, then by Observation \ref{obs-multi}, $PC(G)=n$. Conversely, let $\pi$ be the $PC(G)$-partition, and let $C=(x,y,z,t)$ be an induced $4$-cycle in $G$.  We consider two cases.
		
		\textbf{Case 1.} The set $\lbrace x\rbrace$ has a $pc$-partner in $\lbrace \lbrace y\rbrace,\lbrace t\rbrace \rbrace$. By symmetry, we assume that $\lbrace x\rbrace$ and $\lbrace y\rbrace$ are $pc$-partners. Let $A=N(x) \setminus \lbrace y,t\rbrace$ and $B=N(y) \setminus \lbrace x,z\rbrace$. Observe that $A \cap B =\emptyset$.
		Consider the following subcases.
		
		\textbf{Subcase 1.1.} $A=\emptyset$ and $B=\emptyset$. It follows that  $G \simeq C_4$, as desired.
		
		\textbf{Subcase 1.2.} $A \neq \emptyset$ and $B=\emptyset$. Let $v \in A$. Observe that $v$ has no neighbor in $A \cup \lbrace t,y\rbrace$. Thus, $\lbrace v\rbrace$ must have a $pc$-partner in $\lbrace \lbrace x\rbrace,\lbrace z\rbrace \rbrace$, implying that $N(v)=\lbrace x,z\rbrace$. Hence, $[A,\lbrace x,z\rbrace]$ is full and $[A, \lbrace y,t\rbrace]$ is empty, and so $G$ is a complete bipartite graph with partite sets $A \cup \lbrace y,t\rbrace$ and $\lbrace x,z\rbrace$.
		
		\textbf{Subcase 1.3.} $A = \emptyset$ and $B \neq\emptyset$. Similar to the previous subcase, we can show  that $G$ is a complete bipartite graph with partite sets $B \cup \lbrace x,z\rbrace$ and $\lbrace y,t\rbrace$.
		
		\textbf{Subcase 1.4.} $A \neq \emptyset$ and $B \neq\emptyset$. Note
		that for each $v\in A$,  and for each $\lbrace u\rbrace \in \pi$, if $\lbrace u\rbrace$ is a $pc$-partner of  $\lbrace v\rbrace$, then $\lbrace u\rbrace \subset B \cup \lbrace x,z\rbrace$, implying that $N(v)=B \cup \lbrace x,z\rbrace$. Also, for each $v\in B$, and for each  $\lbrace u\rbrace \in \pi$, if $\lbrace u\rbrace$ is a $pc$-partner of  $\lbrace v\rbrace$, then $\lbrace u\rbrace \subset A \cup \lbrace y,t\rbrace$, implying that $N(u)=A \cup \lbrace y,t\rbrace$. Hence, $[A,B]$ is full, $[A,\lbrace x,z\rbrace]$ is full and $[B, \lbrace y,t\rbrace]$ is full, and so $G$ is a complete bipartite graph with partite sets $A \cup \lbrace y,t\rbrace$ and $B \cup \lbrace x,z\rbrace$.
		
		\textbf{Case 2.} The set $\lbrace x\rbrace$ has no $pc$-partner in  $\lbrace \lbrace y\rbrace,\lbrace t\rbrace \rbrace$. Let $\lbrace e\rbrace  \subset V(G) \setminus \lbrace y,z,t\rbrace$ be a $pc$-partner of $\lbrace x\rbrace$. Observe that the vertices $z$ and $e$ are adjacent. Now let $A=N(x) \setminus \lbrace y,t,e\rbrace$ and $B=N(e) \setminus \lbrace x,z\rbrace$. Observe that  $A \cap B =\emptyset$.     We consider the following subcases.
		
		\textbf{Subcase 2.1.} $A =\emptyset$ and $B =\emptyset$. It follows that $G \simeq K_{2,3}$, as desired.
		
		\textbf{Subcase 2.2.} $A \neq \emptyset$ and $B= \emptyset$. Note that for each $v\in A$, the set $\lbrace v\rbrace$ must have a $pc$-partner in $\lbrace \lbrace x\rbrace,\lbrace z \rbrace \rbrace$, implying that $N(v)= \lbrace x,z\rbrace$.  Hence, $[A,\lbrace x,z\rbrace]$ is full, and so $G$ is a complete bipartite graph with partite sets $A \cup \lbrace y,t,e\rbrace$ and $\lbrace x,z\rbrace$.
		
		\textbf{Subcase 2.3.} $A = \emptyset$ and $B \neq  \emptyset$.
		Note that for each $u\in B$, the set $\lbrace u\rbrace$ must have a $pc$-partner in $\lbrace \lbrace y\rbrace,\lbrace t\rbrace,\lbrace e\rbrace \rbrace$, implying that $N(u)= \lbrace y,t,e \rbrace$.  Hence, $[B,\lbrace y,t,e \rbrace]$ is full, and so $G$ is a complete bipartite graph with partite sets $B \cup \lbrace x,z\rbrace$ and $\lbrace y,t,e\rbrace$.
		
		\textbf{Subcase 2.4.} $A \neq \emptyset$ and $B \neq \emptyset$.
		Note
		that for each $v\in A$,  and for each $\lbrace u\rbrace \in \pi$, if $\lbrace u\rbrace$ is a $pc$-partner of $\lbrace v\rbrace$, then $\lbrace u\rbrace \subset B \cup \lbrace x,z\rbrace$, implying that $N(v)=B \cup \lbrace x,z\rbrace$. Also, for each $v\in B$, and for each $\lbrace u\rbrace \in \pi$, if $\lbrace u\rbrace$ is a $pc$-partner of $\lbrace v\rbrace$, then $\lbrace u\rbrace \subset A \cup \lbrace y,t,e\rbrace$, implying that $N(v)=A \cup \lbrace y,t,e\rbrace$. Hence, $[A,B]$ is full, $[A,\lbrace x,z\rbrace]$ is full and $[B, \lbrace y,t,e\rbrace]$ is full, and so $G$ is a complete bipartite graph with partite sets $A \cup \lbrace y,t,e \rbrace$ and $B \cup \lbrace x,z\rbrace$.
\end{proof}

Applying Corollary \ref{tree-pc-n}, Lemma \ref{n-girth} and Theorem \ref{the-girth4}, we have the following result.
\begin{corollary}
	Let $G$ be a triangle-free graph of order $n\geq 2$. Then $PC(G)=n$ if and only if $G$ is a complete bipartite graph.
\end{corollary}
\subsection{Unicyclic graphs $G$ with $PC(G)=n-2$}
 Throughout this subsection, we will refer to the following trivial remark frequently.
\begin{rem} \label{rem4} 
	Let $G$ be a graph of order $4$. If $G$  contains an induced subgraph of order $3$, whose edge set is empty, then $G$  has no perfect matching. Also, If $G$ contains at least one isolated vertices, then $G$ has no perfect matching.
\end{rem}
\begin{theorem} \label{girth6} 
	If $G$ is a  graph of order $n$ with $g(G)\geq 6$, then $PC(G) < n-2$.
\end{theorem}
\begin{proof}
	By Lemma \ref{n-girth},  $PC(G) \neq n$, so it remains to show that $PC(G) \neq n-2$.
Suppose, to the contrary, that $PC(G)=n-2$. Let $\pi$ be a $PC(G)$-partition. We consider two cases.

\textbf{Case 1.} $\pi$ consists of two doubleton sets and $n-4$ singleton sets. Note that no singleton set can form a paired coalition with a doubleton set. Let $\lbrace v\rbrace$ and $\lbrace u\rbrace$ be two singleton sets in $\pi$ that are $pc$-partners. It follows that the edge $uv$ is a dominating edge in $G$. On the other hand, since $g(G)\geq 6$, $G$ has no dominating edge, a contradiction.

\textbf{Case 2.} $\pi$ consists of a  set of order $3$ and $n-3$ singleton sets.
Let $\lbrace x,y,z\rbrace \in \pi$, and let $C\subseteq G$ be a cycle in $G$ of order $g(G)$. Since $G$ has no dominating edge, no two singleton sets are $pc$-partners. Thus, each singleton set forms a paired coalition with $\lbrace x,y,z\rbrace$. Now we consider the following subcases.

\textbf{Subcase 2.1.} $\lbrace x,y,z\rbrace \cap V(C) =\emptyset$. Since $g(G)\geq 6$, it follows that there exists a vertex $v \in V(C)$ having no neighbor in $\lbrace x,y,z\rbrace$. Now by Remark \ref{rem4}, the sets $\lbrace v\rbrace$ and $\lbrace x,y,z\rbrace$ are not $pc$-partners, a contradiction.

\textbf{Subcase 2.2.} $\lvert \lbrace x,y,z\rbrace \cap V(C)\rvert =1$. By symmetry, we may assume that $z \in V(C)$. Since each of the sets $\lbrace x\rbrace$ and $\lbrace y\rbrace$ dominate at most one vertex from $V(C)$, there exists a vertex $v\in V(C)$ such that $N(v) \cap \lbrace x,y,z\rbrace =\emptyset$. Now, by Remark \ref{rem4},the set $\lbrace v\rbrace$ has no $pc$-partner, a contradiction.

\textbf{Subcase 2.3.} $\lvert \lbrace x,y,z\rbrace \cap V(C)\rvert =2$. By symmetry, we may assume that $x \notin V(C)$. If $x$ has no neighbor in $\lbrace y,z\rbrace$, then by Remark \ref{rem4}, $yz \in E(G)$, and so, considering an arbitrary vertex $v \in V(C) \setminus \lbrace y,z\rbrace$ and a perfect matching $M$ in $G[\lbrace x,y,z,v\rbrace]$, we have $yz \in M$, implying that $V(C)\setminus \lbrace y,z\rbrace \subseteq N(x)$, a contradiction. Hence, we may assume, by symmetry, that $xy \in E(G)$, implying that $N_C (x)=\lbrace y\rbrace$. Then, considering an arbitrary vertex $v \in V(C) \setminus \lbrace y,z\rbrace$ and a perfect matching $M$ in $G[\lbrace x,y,z,v\rbrace]$, we have $xy \in M$, implying that $V(C)\setminus \lbrace y\rbrace \subseteq N[z]$, which is again a contradiction.

\textbf{Subcase 2.4.} $\lbrace x,y,z\rbrace \subset V(C)$. Remark \ref{rem4} implies that $E(G[\lbrace x,y,z\rbrace]) \neq \emptyset$. Thus, by symmetry, we may assume that $xy \in E(G[\lbrace x,y,z\rbrace])$. Now if $z$ has a neighbor in $\lbrace x,y\rbrace$, then there exists a vertex $v \in V(C) \setminus \lbrace x,y,z\rbrace$ such that $v$ has no neighbor in $\lbrace x,y,z\rbrace$, a contradiction. Otherwise, considering an arbitrary vertex $v \in V(C) \setminus \lbrace x,y,z\rbrace$ and a perfect matching $M$ in $G[\lbrace x,y,z,v\rbrace]$, we have $xy \in M$, implying that $V(C)\setminus \lbrace x,y\rbrace \subseteq N[z]$, which is again a contradiction.
\end{proof}
We define the graph $\mathcal{B}$ and the families of graphs $\mathcal{B}_1$ and $\mathcal{B}_2$ as shown in Figure \ref{unicyclic1}.
 \begin{figure}[!htbp]
 	\centering
 	\begin{subfigure}{0.2\textwidth}
 		\begin{tikzpicture}[scale=0.12, transform shape]
 			
 \node [draw, shape=circle,fill=black,scale=2] (v1) at  (0,0){};
\node [draw, shape=circle,fill=black,scale=2] (v2) at  (10,0) {};
\node [draw, shape=circle,fill=black,scale=2] (v3) at  (-4,9) {};
\node [draw, shape=circle,fill=black,scale=2] (v4) at  (14,9) {};
\node [draw, shape=circle,fill=black,scale=2] (v5) at  (5,13) {};

 			\draw(v1)--(v2);
 			\draw(v1)--(v3);
 			\draw(v2)--(v4);
 			\draw(v3)--(v5);
 			
 			\draw(v4)--(v5);		
 		\end{tikzpicture}
 		\caption{The graph $\mathcal{B}$}
 		
 	\end{subfigure}
 	\begin{subfigure}{0.3\textwidth}
 		\begin{tikzpicture}[scale=.12, transform shape]

\node [draw, shape=circle,fill=black,scale=2] (v1) at  (0,0){};
\node [draw, shape=circle,fill=black,scale=2] (v2) at  (10,0) {};
\node [draw, shape=circle,fill=black,scale=2] (v3) at  (-4,9) {};
\node [draw, shape=circle,fill=black,scale=2] (v4) at  (14,9) {};
\node [draw, shape=circle,fill=black,scale=2] (v5) at  (5,13) {};

\node [draw, shape=circle,fill=black,scale=2] (v6) at  (-6,15) {};
\node [draw, shape=circle,fill=black,scale=2] (v7) at  (-12,15) {};
\node [scale=6] at (-9,15) {\large $\cdots$};

\draw(v1)--(v2);
\draw(v1)--(v3);
\draw(v2)--(v4);
\draw(v3)--(v5);
\draw(v4)--(v5);

\draw(v3)--(v6);
\draw(v3)--(v7);		 				
 		\end{tikzpicture}
 		\caption{The family $\mathcal{B}_1$}
 		 	\end{subfigure}
 	\begin{subfigure}{0.3\textwidth}
 	\begin{tikzpicture}[scale=.12, transform shape]
\node [draw, shape=circle,fill=black,scale=2] (v1) at  (0,0){};
\node [draw, shape=circle,fill=black,scale=2] (v2) at  (10,0) {};
\node [draw, shape=circle,fill=black,scale=2] (v3) at  (-4,9) {};
\node [draw, shape=circle,fill=black,scale=2] (v4) at  (14,9) {};
\node [draw, shape=circle,fill=black,scale=2] (v5) at  (5,13) {};

\node [draw, shape=circle,fill=black,scale=2] (v6) at  (-6,15) {};
\node [draw, shape=circle,fill=black,scale=2] (v7) at  (-12,15) {};
\node [scale=6] at (-9,15) {\large $\cdots$};

\node [draw, shape=circle,fill=black,scale=2] (v8) at  (16,15) {};
\node [draw, shape=circle,fill=black,scale=2] (v9) at  (22,15) {};
\node [scale=6] at (19,15) {\large $\cdots$};

\draw(v1)--(v2);
\draw(v1)--(v3);
\draw(v2)--(v4);
\draw(v3)--(v5);
\draw(v4)--(v5);

\draw(v3)--(v6);
\draw(v3)--(v7);		 	 		

\draw(v4)--(v8);
\draw(v4)--(v9);		 	 		

 		\end{tikzpicture}
 		\caption{The family $\mathcal{B}_2$}
 	\end{subfigure}
	\caption{Graphs $G$ with $g(G)=5$ and $PC(G)=n-2$}\label{unicyclic1}
\end{figure}
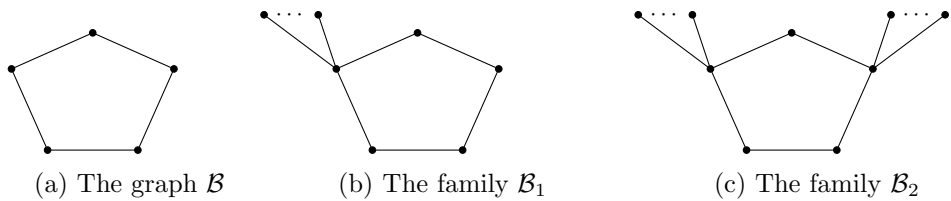

\begin{theorem} \label{girth5}
	Let $G$ be a  graph of order $n$ with $g(G)=5$. Then $PC(G)=n-2$ if and only if $G \in \mathcal{B} \cup \mathcal{B}_1 \cup \mathcal{B}_2$.
\end{theorem}
\begin{figure}[!htbp]
	\centering
	\begin{subfigure}{0.2\textwidth}
		\begin{tikzpicture}[scale=.12, transform shape]
			
			\node [draw, shape=circle,fill=black,scale=2] (v1) at  (0,0){};
			\node [draw, shape=circle,fill=black,scale=2] (v2) at  (10,0) {};
			\node [draw, shape=circle,fill=black,scale=2] (v3) at  (-4,9) {};
			\node [draw, shape=circle,fill=black,scale=2] (v4) at  (14,9) {};
			\node [draw, shape=circle,fill=black,scale=2] (v5) at  (5,13) {};
			
			\node [scale=6] at (-4,10.5) {$x$};
			\node [scale=6] at (5,14.5) {$y$};
			\node [scale=6] at (14,10.5) {$z$};
			
			\draw(v1)--(v2);
			\draw(v1)--(v3);
			\draw(v2)--(v4);
			\draw(v3)--(v5);
			
			\draw(v4)--(v5);		
		\end{tikzpicture}
	\end{subfigure}
	\begin{subfigure}{0.3\textwidth}
		\begin{tikzpicture}[scale=.12, transform shape]
			\node [draw, shape=circle,fill=black,scale=2] (v1) at  (0,0){};
			\node [draw, shape=circle,fill=black,scale=2] (v2) at  (10,0) {};
			\node [draw, shape=circle,fill=black,scale=2] (v3) at  (-4,9) {};
			\node [draw, shape=circle,fill=black,scale=2] (v4) at  (14,9) {};
			\node [draw, shape=circle,fill=black,scale=2] (v5) at  (5,13) {};

			\node [draw, shape=circle,fill=black,scale=2] (v6) at  (-6,15) {};
			\node [draw, shape=circle,fill=black,scale=2] (v7) at  (-12,15) {};
			\node [scale=6] at (-9,15) {\large $\cdots$};
			\node [scale=6] at (-3.5,10.5) {$x$};
			\node [scale=6] at (5,14.5) {$y$};
			\node [scale=6] at (14,10.5) {$z$};
			
			\draw(v1)--(v2);
			\draw(v1)--(v3);
			\draw(v2)--(v4);
			\draw(v3)--(v5);
			\draw(v4)--(v5);
			
			\draw(v3)--(v6);
			\draw(v3)--(v7);		 				
		\end{tikzpicture}
	\end{subfigure}
	\begin{subfigure}{0.3\textwidth}
		\begin{tikzpicture}[scale=.12, transform shape]
			\node [draw, shape=circle,fill=black,scale=2] (v1) at  (0,0){};
			\node [draw, shape=circle,fill=black,scale=2] (v2) at  (10,0) {};
			\node [draw, shape=circle,fill=black,scale=2] (v3) at  (-4,9) {};
			\node [draw, shape=circle,fill=black,scale=2] (v4) at  (14,9) {};
			\node [draw, shape=circle,fill=black,scale=2] (v5) at  (5,13) {};
			
			\node [draw, shape=circle,fill=black,scale=2] (v6) at  (-6,15) {};
			\node [draw, shape=circle,fill=black,scale=2] (v7) at  (-12,15) {};
			\node [scale=6] at (-9,15) {\large $\cdots$};

			\node [draw, shape=circle,fill=black,scale=2] (v8) at  (16,15) {};
			\node [draw, shape=circle,fill=black,scale=2] (v9) at  (22,15) {};
			\node [scale=6] at (19,15) {\large $\cdots$};
			
			\node [scale=6] at (-3.5,10.5) {$x$};
			\node [scale=6] at (5,14.5) {$y$};
			\node [scale=6] at (13.5,10.5) {$z$};
			\draw(v1)--(v2);
			\draw(v1)--(v3);
			\draw(v2)--(v4);
			\draw(v3)--(v5);
			\draw(v4)--(v5);
			
			\draw(v3)--(v6);
			\draw(v3)--(v7);		 	 		
			
			\draw(v4)--(v8);
			\draw(v4)--(v9);		 	 		
			
		\end{tikzpicture}
	\end{subfigure}
	\caption{$pc$-partitions of order $n-2$}\label{unicyclic2}
\end{figure}
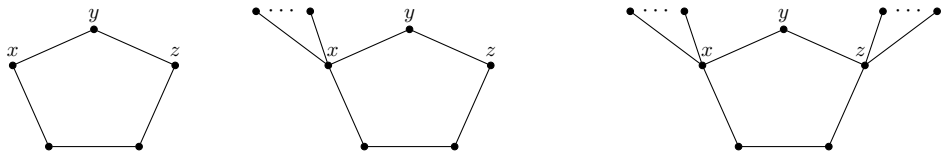
\begin{proof}
	 In Figure \ref{unicyclic2}, a $pc$-partition $\pi$ of order $n-2$ is illustrated for  graphs $G \in \mathcal{B} \cup \mathcal{B}_1 \cup \mathcal{B}_2$, where $\pi$ consists of the set $\lbrace x,y,z\rbrace$ and $n-3$ singleton sets. Conversely,
let $PC(G)=n-2$, and let $\pi$ be a $PC(G)$-partition. We consider two cases.

\textbf{Case 1.} $\pi$ consists of two doubleton sets and $n-4$ singleton sets. Note that no singleton set can form a paired coalition with a doubleton set. Let $\lbrace v\rbrace$ and $\lbrace u\rbrace$ be two singleton sets in $\pi$ that are $pc$-partners. It follows that the edge $uv$ is a dominating edge in $G$. On the other hand, since $g(G)= 5$, $G$ has no dominating edge, a contradiction.

\textbf{Case 2.} $\pi$ consists of a  set of order $3$ and $n-3$ singleton sets.
Let $\lbrace x,y,z\rbrace \in \pi$, and let $C\subseteq G$ be a $5$-cycle in $G$. Since $G$ has no dominating edge, no two singleton sets are $pc$-partners. Thus, each singleton set forms a paired coalition with $\lbrace x,y,z\rbrace$. Now we consider the following subcases.

\textbf{Subcase 2.1.} $\lbrace x,y,z\rbrace \cap V(C) =\emptyset$. Since each vertex in $\lbrace x,y,z\rbrace$ has at most one neighbor in $V(C)$, we can find a vertex $v \in V(C)$ having no neighbor in $\lbrace x,y,z\rbrace$. Now  the sets $\lbrace v\rbrace$ and $\lbrace x,y,z\rbrace$ are not $pc$-partners, so this subcase never occurs. 

\textbf{Subcase 2.2.} $\lvert \lbrace x,y,z\rbrace \cap V(C)\rvert =1$. By symmetry, we may assume that $z \in V(C)$. If $xy \in E(G)$, then the set $\lbrace x,y\rbrace$ does not dominate all vertices of $V(C) \setminus N_C [z]$, for otherwise, a $3$-cycle or a $4$-cycle would be created. Therefore, if $xy \in E(G)$, then  there exists a vertex $v \in V(C)$ having no neighbor in $\lbrace x,y,z\rbrace$, which is impossible. Hence, we may assume that $xy \notin E(G)$. Now Remark \ref{rem4} implies that $z$ has a neighbor in $\lbrace x,y\rbrace$. By symmetry, we may assume that $xz \in E(G)$, implying that $x$ has no neighbor in $V(C) \setminus N_C [z]$. Note also that the set $\lbrace y\rbrace$ does not dominate $V(C) \setminus N_C [z]$. Therefore,  we can again find a vertex $v \in V(C)$ having no neighbor in $\lbrace x,y,z\rbrace$. Hence, this subcase never occurs.

\textbf{Subcase 2.3.} $\lvert \lbrace x,y,z\rbrace \cap V(C)\rvert =2$. By symmetry, we may assume that $x \notin V(C)$. If $x$ has no neighbor in $\lbrace y,z\rbrace$, then $yz \in E(G)$. Now considering an arbitrary vertex $v \in V(C) \setminus \lbrace y,z\rbrace$ and a perfect matching $M$ in $G[\lbrace x,y,z,v\rbrace]$, we have $yz \in M$, implying that $V(C) \setminus \lbrace y,z\rbrace \subseteq N(x)$, which is impossible. Hence, we may assume, by symmetry, that $xy \in E(G)$, implying that $N_C (x)=\lbrace y\rbrace$. Then, considering an arbitrary vertex $v \in V(C) \setminus \lbrace y,z\rbrace$ and a perfect matching $M$ in $G[\lbrace x,y,z,v\rbrace]$, we have $xy \in M$, implying that $V(C)\setminus \lbrace y\rbrace \subseteq N[z]$, which is again impossible. Hence, this subcase never occurs.

\textbf{Subcase 2.4.} $\lbrace x,y,z\rbrace \subset V(C)$. Note that $E(G[\lbrace x,y,z\rbrace]) \neq \emptyset$. By symmetry, we may assume that $xy \in E(G[\lbrace x,y,z\rbrace])$. If $z$ has no neighbor in $\lbrace x,y\rbrace$, then Remark \ref{rem4} implies that the vertices in $V(C) \setminus \lbrace z\rbrace$ have no neighbor in $V(G) \setminus V(C)$, and so $G \in \mathcal{B} \cup \mathcal{B}_1$.  Otherwise, by symmetry, we may assume that $xz \in E(G)$. Now, by Remark \ref{rem4},  the vertices in $V(C) \setminus \lbrace z,y\rbrace$ have no neighbor in $V(G) \setminus V(C)$, implying that $G \in \mathcal{B} \cup \mathcal{B}_1 \cup \mathcal{B}_2$. This completes the proof.

\end{proof}

We define  the families of graphs $\mathcal{D}_1$ and $\mathcal{D}_2$ as shown in Figure \ref{unicyclic3}.
 \begin{figure}[!htbp]
	\centering

	\begin{subfigure}{0.18\textwidth}
		\begin{tikzpicture}[scale=.15, transform shape]
\node [draw, shape=circle,fill=black,scale=2] (v1) at  (0,0) {};
\node [draw, shape=circle,fill=black,scale=2] (v2) at  (0,10) {};
\node [draw, shape=circle,fill=black,scale=2] (v3) at  (10,0) {};
\node [draw, shape=circle,fill=black,scale=2] (v4) at  (10,10) {};
\node [draw, shape=circle,fill=black,scale=2] (v5) at  (14,14) {};
\node [draw, shape=circle,fill=black,scale=2] (v6) at  (20,14) {};
\node [scale=6] at (17,14) {\large $\cdots$};
\draw(v1)--(v2);
\draw(v1)--(v3);
\draw(v2)--(v4);
\draw(v3)--(v4);

\draw(v4)--(v5);
\draw(v4)--(v6);

		\end{tikzpicture}
		\caption{The family $\mathcal{D}_1$}
	\end{subfigure}
	\begin{subfigure}{0.3\textwidth}
		\begin{tikzpicture}[scale=.15, transform shape]
\node [draw, shape=circle,fill=black,scale=2] (v1) at  (0,0) {};
\node [draw, shape=circle,fill=black,scale=2] (v2) at  (0,10) {};
\node [draw, shape=circle,fill=black,scale=2] (v3) at  (10,0) {};
\node [draw, shape=circle,fill=black,scale=2] (v4) at  (10,10) {};
\node [draw, shape=circle,fill=black,scale=2] (v5) at  (14,14) {};
\node [draw, shape=circle,fill=black,scale=2] (v6) at  (20,14) {};
\node [scale=6] at (17,14) {\large $\cdots$};
\node [draw, shape=circle,fill=black,scale=2] (v7) at  (-4,4) {};
\node [draw, shape=circle,fill=black,scale=2] (v8) at  (-10,4) {};
\node [scale=6] at (-7,4) {\large $\cdots$};
\draw(v1)--(v2);
\draw(v1)--(v3);
\draw(v2)--(v4);
\draw(v3)--(v4);

\draw(v4)--(v5);
\draw(v4)--(v6);			

\draw(v1)--(v7);
\draw(v1)--(v8);			
		\end{tikzpicture}
		\caption{The family $\mathcal{D}_2$}
	\end{subfigure}
\caption{Unicyclic graphs $G$ with $g(G)=4$ and $PC(G)=n-2$}\label{unicyclic3}
\end{figure}
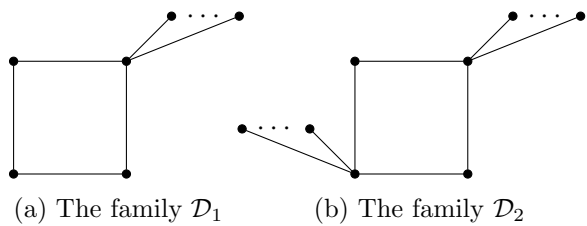

		\begin{lemma} \label{girth4}
			Let $G$ be a unicyclic  graph of order $n$ with $g(G)=4$. Then $PC(G)=n-2$ if and only if $G \in \mathcal{D}_1 \cup \mathcal{D}_2$.
		\end{lemma}
	 \begin{figure}[!htbp]
		\centering
		
		\begin{subfigure}{0.18\textwidth}
			\begin{tikzpicture}[scale=.15, transform shape]
\node [draw, shape=circle,fill=black,scale=2] (v1) at  (0,0) {};
\node [draw, shape=circle,fill=black,scale=2] (v2) at  (0,10) {};
\node [draw, shape=circle,fill=black,scale=2] (v3) at  (10,0) {};
\node [draw, shape=circle,fill=black,scale=2] (v4) at  (10,10) {};
\node [draw, shape=circle,fill=black,scale=2] (v5) at  (14,14) {};
\node [draw, shape=circle,fill=black,scale=2] (v6) at  (20,14) {};
\node [scale=6] at (17,14) {\large $\cdots$};
				
				\node [scale=6] at (0,-1.5) {$x$};
				\node [scale=6] at (0,11.5) {$y$};
				\node [scale=6] at (10,11.5) {$z$};
				
				\draw(v1)--(v2);
				\draw(v1)--(v3);
				\draw(v2)--(v4);
				\draw(v3)--(v4);
				
				\draw(v4)--(v5);
				\draw(v4)--(v6);
				
			\end{tikzpicture}

		\end{subfigure}
		\begin{subfigure}{0.3\textwidth}
			\begin{tikzpicture}[scale=.15, transform shape]
\node [draw, shape=circle,fill=black,scale=2] (v1) at  (0,0) {};
\node [draw, shape=circle,fill=black,scale=2] (v2) at  (0,10) {};
\node [draw, shape=circle,fill=black,scale=2] (v3) at  (10,0) {};
\node [draw, shape=circle,fill=black,scale=2] (v4) at  (10,10) {};
\node [draw, shape=circle,fill=black,scale=2] (v5) at  (14,14) {};
\node [draw, shape=circle,fill=black,scale=2] (v6) at  (20,14) {};
\node [scale=6] at (17,14) {\large $\cdots$};
\node [draw, shape=circle,fill=black,scale=2] (v7) at  (-4,4) {};
\node [draw, shape=circle,fill=black,scale=2] (v8) at  (-10,4) {};
\node [scale=6] at (-7,4) {\large $\cdots$};
				
				\node [scale=6] at (0,-1.5) {$x$};
				\node [scale=6] at (0,11.5) {$y$};
				\node [scale=6] at (10,11.5) {$z$};
				\draw(v1)--(v2);
				\draw(v1)--(v3);
				\draw(v2)--(v4);
				\draw(v3)--(v4);
				
				\draw(v4)--(v5);
				\draw(v4)--(v6);			
				
				\draw(v1)--(v7);
				\draw(v1)--(v8);			
			\end{tikzpicture}

		\end{subfigure}
	\caption{$pc$-partitions of order $n-2$}\label{unicyclic4}	
\end{figure}
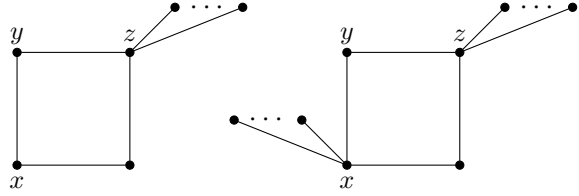
	\begin{proof}
		In Figure \ref{unicyclic4}, a $pc$-partition $\pi$ of order $n-2$ is illustrated for  graphs $G \in \mathcal{D}_1 \cup \mathcal{D}_2$, where $\pi$ consists of the set $\lbrace x,y,z\rbrace$ and $n-3$ singleton sets. Conversely,
let $PC(G)=n-2$, and let $\pi$ be a $PC(G)$-partition. Further, let $C =(v_1,v_2,v_3,v_4)$ be the cycle in $G$. The result can be readily verified for $n=5$, so we assume that $n\geq 6$. Consider two cases.

\textbf{Case 1.} $\pi$ consists of two doubleton sets and $n-4$ singleton sets. Note that no singleton set can form a paired coalition with a doubleton set. Let $\lbrace v\rbrace$ and $\lbrace u\rbrace$ be two singleton sets in $\pi$ that are $pc$-partners. It follows that the edge $uv$ is a dominating edge in $G$, and so $uv \in E(C)$. Suppose, by symmetry, that $u=v_1$ and $v=v_2$. Now we  define $A=N(u) \setminus \lbrace v_2,v_4\rbrace$ and $B=N(v) \setminus \lbrace v_1,v_3\rbrace$. Let $V_1, V_2 \subset \pi$, where $\lvert V_1\rvert =\lvert V_2\rvert =2$. Note that the sets $[A,B]$, $[A,\lbrace v_3,v_4\rbrace]$ and $[B,\lbrace v_3,v_4\rbrace]$ are empty sets, and so $E(G[V \setminus \lbrace v_1,v_2\rbrace])=\lbrace v_3 v_4\rbrace$. Thus, $G[V_1 \cup V_2]$ has no perfect matching, so this case is impossible.

\textbf{Case 2.} $\pi$ consists of a  set of order $3$ and $n-3$ singleton sets. First assume that  there exist two singleton sets in $\pi$ that are $pc$-partners. Let $\lbrace u\rbrace$ and $\lbrace v\rbrace$ be $pc$-partners in $\pi$. From our discussion in the previous case, we may assume that $u=v_1$ and $v=v_2$. Define $A=N(u) \setminus \lbrace v_2,v_4\rbrace$ and $B=N(v) \setminus \lbrace v_1,v_3\rbrace$. Let $\lbrace x,y,z\rbrace \in \pi$. Since the sets $[A,B]$, $[A,\lbrace v_3,v_4\rbrace]$ and $[B,\lbrace v_3,v_4\rbrace]$ are empty sets, it follows from Remark \ref{rem4} that $\lbrace v_3,v_4\rbrace \subset \lbrace x,y,z\rbrace$. By symmetry, we may assume that $\lbrace v_3,v_4\rbrace =\lbrace x,y\rbrace$, and that $z \in A$. Further, it is easy to verify that the edge $v_1 v_2$ is the only dominating edge in $G$, implying that all singleton sets in $\pi \setminus \lbrace \lbrace u\rbrace ,\lbrace v\rbrace \rbrace$ must form a paired coalition with $\lbrace x,y,z\rbrace$. Hence, $B= \emptyset$ and $A=\lbrace z\rbrace$, and so $G \in \mathcal{D}_1$. Now we assume that no two singleton sets in $\pi$ are $pc$-partners, implying that all singleton sets in $\pi$ form a paired coalition with $\lbrace x,y,z\rbrace$. We consider the following subcases.

\textbf{Subcase 2.1.} $\lbrace x,y,z\rbrace \cap V(C) =\emptyset$.
Since each vertex in $\lbrace x,y,z\rbrace$ has at most one neighbor in $V(C)$, we can find a vertex $v \in V(C)$ having no neighbor in $\lbrace x,y,z\rbrace$. Therefore, the sets $\lbrace v\rbrace$ and $\lbrace x,y,z\rbrace$ are not $pc$-partners, a contradiction.

\textbf{Subcase 2.2.} $\lvert \lbrace x,y,z\rbrace \cap V(C)\rvert =1$. By symmetry, we may assume that $z \in V(C)$. If $xy \in E(G)$, then $z$ has no neighbor in $\lbrace x,y\rbrace$, for otherwise, there would be a vertex $v \in V(C)$ having no neighbor in $\lbrace x,y,z\rbrace$. Therefore, if $xy \in E(G)$, then  considering an arbitrary vertex $v \in V(C) \setminus \lbrace z\rbrace$ and a perfect matching $M$ in $G[\lbrace x,y,z,v\rbrace]$, we have $xy \in M$, implying that $V(C) \subseteq N[z]$, a contradiction. Hence, we may assume that $xy \notin E(G)$, implying that $z$ has a neighbor in $\lbrace x,y\rbrace$. By symmetry, we may assume that $xz \in E(G)$. Note that $N_C (x)=\lbrace z\rbrace$. Thus, considering an arbitrary vertex $v \in V(C) \setminus \lbrace z\rbrace$ and a perfect matching $M$ in $G[\lbrace x,y,z,v\rbrace]$, we have $xz \in M$, implying that $V(C)\setminus \lbrace z\rbrace \subseteq N(y)$, which is again a contradiction.

\textbf{Subcase 2.3.} $\lvert \lbrace x,y,z\rbrace \cap V(C)\rvert =2$. By symmetry, we may assume that $x \notin V(C)$. If $x$ has no neighbor in $\lbrace y,z\rbrace$, then we can find a vertex $v \in V(C) \setminus \lbrace y,z\rbrace$ such that $x$ has no neighbor in $\lbrace v,y,z\rbrace$, a contradiction.
Hence, we may assume that $xy \in E(G)$. Note that $N_C (x)=\lbrace y\rbrace$. Let $y=v_1$. Now considering an arbitrary vertex $v \in V(C) \setminus \lbrace y,z\rbrace$ and a perfect matching $M$ in $G[\lbrace x,y,z,v\rbrace]$, we have $xy \in M$, implying that $V(C)\setminus \lbrace y\rbrace \subseteq N[z]$. Hence, $z=v_3$. Further, by Remark \ref{rem4}, $N(v_2)=N(v_4)=\lbrace y,z\rbrace$. Hence, $G \in \mathcal{D}_1 \cup \mathcal{D}_2$.

\textbf{Subcase 2.4.} $\lbrace x,y,z\rbrace \subset V(C)$. By symmetry, we may assume that $N_C (y)=\lbrace x,z\rbrace$. Let $\lbrace v_1\rbrace =V(C) \setminus \lbrace x,y,z\rbrace$.  It is easy to verify that $N(y) =N_C (y)= \lbrace x,z\rbrace$. Further, by Remark \ref{rem4}, $N_G (v_1)=N_C (v_1)$ and $V(G)=V(C) \cup N(x) \cup N(z)$. Hence, $G \in \mathcal{D}_1 \cup \mathcal{D}_2$.
\end{proof}

We define  the families of graphs $\mathcal{E}_1$,$\mathcal{E}_2$, \dots ,$\mathcal{E}_7$  as shown in Figure \ref{unicyclic5}.
 \begin{figure}[!htbp]
	\centering
	\begin{subfigure}{0.32\textwidth}
		\begin{tikzpicture}[scale=.15, transform shape]
			
			\node [draw, shape=circle,fill=black,scale=2] (v1) at  (0,0) {};
			\node [draw, shape=circle,fill=black,scale=2] (v2) at  (10,0) {};
			\node [draw, shape=circle,fill=black,scale=2] (v3) at  (5,9) {};
			
			\node [draw, shape=circle,fill=black,scale=2] (v4) at  (-4,6) {};
			\node [draw, shape=circle,fill=black,scale=2] (v5) at  (-10,6) {};
			\node [scale=6] at (-7,6) {\large $\cdots$};
			\node [draw, shape=circle,fill=black,scale=2] (v6) at  (14,6) {};
			\node [draw, shape=circle,fill=black,scale=2] (v7) at  (20,6) {};
			\node [scale=6] at (17,6) {\large $\cdots$};
			\draw(v1)--(v2);
			\draw(v1)--(v3);
			\draw(v2)--(v3);
			
			\draw(v1)--(v4);
\draw(v1)--(v5);
\draw(v2)--(v6);
\draw(v2)--(v7);			
			
		\end{tikzpicture}
		\caption{The family $\mathcal{E}_1$}
		
	\end{subfigure}
	\begin{subfigure}{0.32\textwidth}
		\begin{tikzpicture}[scale=.15, transform shape]
			
			\node [draw, shape=circle,fill=black,scale=2] (v1) at  (0,0) {};
\node [draw, shape=circle,fill=black,scale=2] (v2) at  (10,0) {};
\node [draw, shape=circle,fill=black,scale=2] (v3) at  (5,9) {};

\node [draw, shape=circle,fill=black,scale=2] (v4) at  (-4,6) {};
\node [draw, shape=circle,fill=black,scale=2] (v5) at  (-10,6) {};
\node [scale=6] at (-7,6) {\large $\cdots$};
\node [draw, shape=circle,fill=black,scale=2] (v6) at  (14,6) {};
\node [draw, shape=circle,fill=black,scale=2] (v7) at  (20,6) {};
\node [scale=6] at (17,6) {\large $\cdots$}; {};
			\node [draw, shape=circle,fill=black,scale=2] (v8) at  (2,15) {};
\node [draw, shape=circle,fill=black,scale=2] (v9) at  (8,15) {};
\node [scale=6] at (5,15) {\large $\cdots$}; {};			
			\draw(v1)--(v2);
			\draw(v1)--(v3);
			\draw(v2)--(v3);
			
			\draw(v1)--(v4);
			\draw(v1)--(v5);
			\draw(v2)--(v6);
			\draw(v2)--(v7);			
			\draw(v3)--(v8);
\draw(v3)--(v9);			
			
		\end{tikzpicture}
		\caption{The family $\mathcal{E}_2$}	
	\end{subfigure}
	\begin{subfigure}{0.32\textwidth}
	\begin{tikzpicture}[scale=.15, transform shape]
		
		\node [draw, shape=circle,fill=black,scale=2] (v1) at  (0,0) {};
		\node [draw, shape=circle,fill=black,scale=2] (v2) at  (10,0) {};
		\node [draw, shape=circle,fill=black,scale=2] (v3) at  (5,9) {};
		
		\node [draw, shape=circle,fill=black,scale=2] (v4) at  (20,0) {};
		\node [draw, shape=circle,fill=black,scale=2] (v5) at  (24,6) {};
		\node [draw, shape=circle,fill=black,scale=2] (v6) at  (30,6) {};
		\node [scale=6] at (27,6) {\large $\cdots$}; {};
		\draw(v1)--(v2);
		\draw(v1)--(v3);
		\draw(v2)--(v3);
		
		\draw(v2)--(v4);
		\draw(v4)--(v5);
		\draw(v4)--(v6);
	\end{tikzpicture}
	\caption{The family $\mathcal{E}_3$}
	
\end{subfigure}
	
	\begin{subfigure}{0.49\textwidth}
	\begin{tikzpicture}[scale=.15, transform shape]
		
		\node [draw, shape=circle,fill=black,scale=2] (v1) at  (0,0) {};
		\node [draw, shape=circle,fill=black,scale=2] (v2) at  (10,0) {};
		\node [draw, shape=circle,fill=black,scale=2] (v3) at  (5,9) {};
		
		\node [draw, shape=circle,fill=black,scale=2] (v4) at  (20,0) {};
		\node [draw, shape=circle,fill=black,scale=2] (v5) at  (24,6) {};
		\node [draw, shape=circle,fill=black,scale=2] (v8) at  (30,6) {};
		\node [scale=6] at (27,6) {\large $\cdots$}; {};
		\node [draw, shape=circle,fill=black,scale=2] (v6) at  (-4,6) {};
		\node [draw, shape=circle,fill=black,scale=2] (v7) at  (-10,6) {};
		\node [scale=6] at (-7,6) {\large $\cdots$}; {};
		\draw(v1)--(v2);
		\draw(v1)--(v3);
		\draw(v2)--(v3);
		
		\draw(v2)--(v4);
		\draw(v4)--(v5);
			\draw(v4)--(v8);
		\draw(v1)--(v6);
		
		\draw(v1)--(v7);
	\end{tikzpicture}
	\caption{The family $\mathcal{E}_4$}
	
\end{subfigure}
	\begin{subfigure}{0.49\textwidth}
	\begin{tikzpicture}[scale=.15, transform shape]
		
		\node [draw, shape=circle,fill=black,scale=2] (v1) at  (0,0) {};
		\node [draw, shape=circle,fill=black,scale=2] (v2) at  (10,0) {};
		\node [draw, shape=circle,fill=black,scale=2] (v3) at  (5,9) {};
		
		\node [draw, shape=circle,fill=black,scale=2] (v4) at  (20,0) {};
		\node [draw, shape=circle,fill=black,scale=2] (v5) at  (30,0) {};
		\node [draw, shape=circle,fill=black,scale=2] (v6) at  (34,6) {};
		\node [draw, shape=circle,fill=black,scale=2] (v7) at  (40,6) {};
		\node [scale=6] at (37,6) {\large $\cdots$}; {};
		\draw(v1)--(v2);
		\draw(v1)--(v3);
		\draw(v2)--(v3);
		
		\draw(v2)--(v4);
		\draw(v4)--(v5);
		\draw(v5)--(v6);
		\draw(v5)--(v7);
	\end{tikzpicture}
	\caption{The family $\mathcal{E}_5$}
	
\end{subfigure}
	\begin{subfigure}{0.49\textwidth}
	\begin{tikzpicture}[scale=.15, transform shape]
		
		\node [draw, shape=circle,fill=black,scale=2] (v1) at  (0,0) {};
		\node [draw, shape=circle,fill=black,scale=2] (v2) at  (10,0) {};
		\node [draw, shape=circle,fill=black,scale=2] (v3) at  (5,9) {};
		
		\node [draw, shape=circle,fill=black,scale=2] (v4) at  (20,0) {};
		\node [draw, shape=circle,fill=black,scale=2] (v5) at  (30,0) {};
		\node [draw, shape=circle,fill=black,scale=2] (v6) at  (14,6) {};
		\node [draw, shape=circle,fill=black,scale=2] (v7) at  (20,6) {};
		\node [scale=6] at (17,6) {\large $\cdots$}; {};
		\draw(v1)--(v2);
		\draw(v1)--(v3);
		\draw(v2)--(v3);
		
		\draw(v2)--(v4);
		\draw(v4)--(v5);
		\draw(v2)--(v6);
		\draw(v2)--(v7);
	\end{tikzpicture}
	\caption{The family $\mathcal{E}_6$}
	
\end{subfigure}
	\begin{subfigure}{0.49\textwidth}
	\begin{tikzpicture}[scale=.15, transform shape]
		
		\node [draw, shape=circle,fill=black,scale=2] (v1) at  (0,0) {};
		\node [draw, shape=circle,fill=black,scale=2] (v2) at  (10,0) {};
		\node [draw, shape=circle,fill=black,scale=2] (v3) at  (5,9) {};
		
		\node [draw, shape=circle,fill=black,scale=2] (v4) at  (20,0) {};
		\node [draw, shape=circle,fill=black,scale=2] (v5) at  (30,0) {};
		\node [draw, shape=circle,fill=black,scale=2] (v6) at  (14,6) {};
		\node [draw, shape=circle,fill=black,scale=2] (v7) at  (20,6) {};
			\node [scale=6] at (17,6) {\large $\cdots$}; {};
				\node [draw, shape=circle,fill=black,scale=2] (v8) at  (34,6) {};
		\node [draw, shape=circle,fill=black,scale=2] (v9) at  (40,6) {};
			\node [scale=6] at (37,6) {\large $\cdots$}; {};
		\draw(v1)--(v2);
		\draw(v1)--(v3);
		\draw(v2)--(v3);
		
		\draw(v2)--(v4);
		\draw(v4)--(v5);
		\draw(v2)--(v6);
		\draw(v2)--(v7);
		\draw(v5)--(v8);
\draw(v5)--(v9);
	\end{tikzpicture}
	\caption{The family $\mathcal{E}_7$}
	
\end{subfigure}
\caption{Unicyclic graphs $G$ with $g(G)=3$ and $PC(G)=n-2$}\label{unicyclic5}
\end{figure}
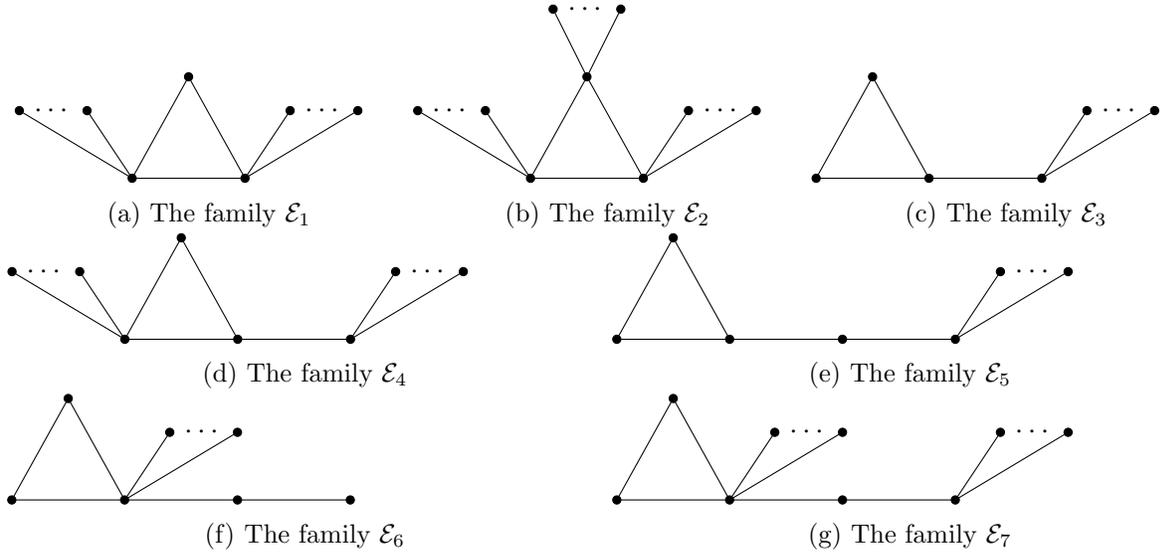
		\begin{lemma} \label{girth3}
	Let $G$ be a unicyclic graph of order $n$ with $g(G)=3$. Then $PC(G)=n-2$ if and only if $G \in \bigcup_{i=1} ^{7} \mathcal{E}_i$.
\end{lemma}
 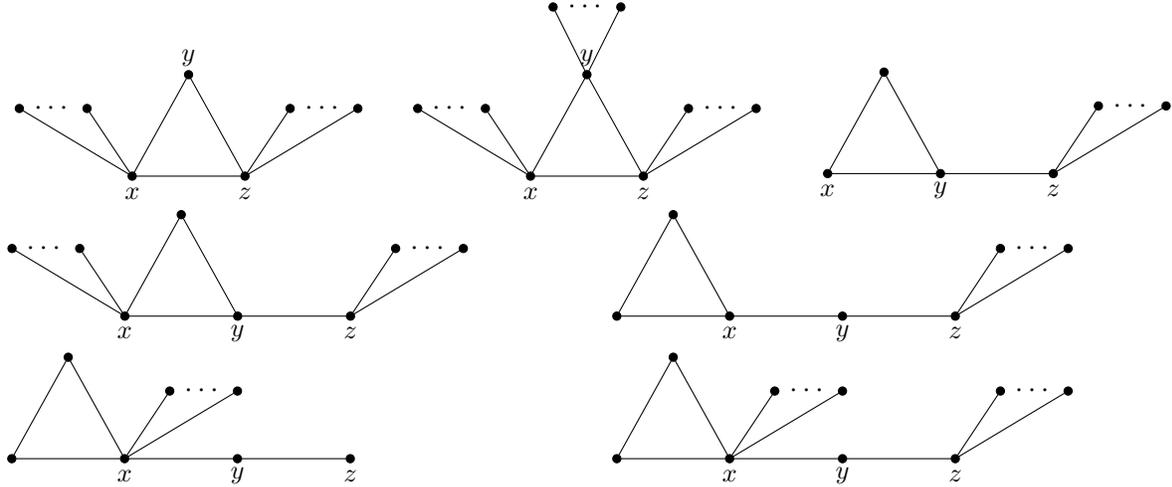
\begin{figure}[!htbp]
	\centering
	\begin{subfigure}{0.32\textwidth}
		\begin{tikzpicture}[scale=.15, transform shape]
			\node [draw, shape=circle,fill=black,scale=2] (v1) at  (0,0) {};
\node [draw, shape=circle,fill=black,scale=2] (v2) at  (10,0) {};
\node [draw, shape=circle,fill=black,scale=2] (v3) at  (5,9) {};

\node [draw, shape=circle,fill=black,scale=2] (v4) at  (-4,6) {};
\node [draw, shape=circle,fill=black,scale=2] (v5) at  (-10,6) {};
\node [scale=6] at (-7,6) {\large $\cdots$};
\node [draw, shape=circle,fill=black,scale=2] (v6) at  (14,6) {};
\node [draw, shape=circle,fill=black,scale=2] (v7) at  (20,6) {};
\node [scale=6] at (17,6) {\large $\cdots$};
			\node [scale=6] at (0,-1.5) {$x$};
			\node [scale=6] at (5,10.5) {$y$};
			\node [scale=6] at (10,-1.5) {$z$};
			\draw(v1)--(v2);
			\draw(v1)--(v3);
			\draw(v2)--(v3);
			
			\draw(v1)--(v4);
			\draw(v1)--(v5);
			\draw(v2)--(v6);
			\draw(v2)--(v7);			
			
		\end{tikzpicture}
		
	\end{subfigure}
	\begin{subfigure}{0.32\textwidth}
		\begin{tikzpicture}[scale=.15, transform shape]
			\node [draw, shape=circle,fill=black,scale=2] (v1) at  (0,0) {};
\node [draw, shape=circle,fill=black,scale=2] (v2) at  (10,0) {};
\node [draw, shape=circle,fill=black,scale=2] (v3) at  (5,9) {};

\node [draw, shape=circle,fill=black,scale=2] (v4) at  (-4,6) {};
\node [draw, shape=circle,fill=black,scale=2] (v5) at  (-10,6) {};
\node [scale=6] at (-7,6) {\large $\cdots$};
\node [draw, shape=circle,fill=black,scale=2] (v6) at  (14,6) {};
\node [draw, shape=circle,fill=black,scale=2] (v7) at  (20,6) {};
\node [scale=6] at (17,6) {\large $\cdots$}; {};
\node [draw, shape=circle,fill=black,scale=2] (v8) at  (2,15) {};
\node [draw, shape=circle,fill=black,scale=2] (v9) at  (8,15) {};
\node [scale=6] at (5,15) {\large $\cdots$}; {};
			\node [scale=6] at (0,-1.5) {$x$};
\node [scale=6] at (5,10.5) {$y$};
\node [scale=6] at (10,-1.5) {$z$};			
			\draw(v1)--(v2);
			\draw(v1)--(v3);
			\draw(v2)--(v3);
			
			\draw(v1)--(v4);
			\draw(v1)--(v5);
			\draw(v2)--(v6);
			\draw(v2)--(v7);			
			\draw(v3)--(v8);
			\draw(v3)--(v9);			
			
		\end{tikzpicture}
	\end{subfigure}
	\begin{subfigure}{0.32\textwidth}
		\begin{tikzpicture}[scale=.15, transform shape]
				\node [draw, shape=circle,fill=black,scale=2] (v1) at  (0,0) {};
		\node [draw, shape=circle,fill=black,scale=2] (v2) at  (10,0) {};
		\node [draw, shape=circle,fill=black,scale=2] (v3) at  (5,9) {};
		
		\node [draw, shape=circle,fill=black,scale=2] (v4) at  (20,0) {};
		\node [draw, shape=circle,fill=black,scale=2] (v5) at  (24,6) {};
		\node [draw, shape=circle,fill=black,scale=2] (v6) at  (30,6) {};
		\node [scale=6] at (27,6) {\large $\cdots$}; {};
				\node [scale=6] at (0,-1.5) {$x$};

	\node [scale=6] at (10,-1.5) {$y$};	
		\node [scale=6] at (20,-1.5) {$z$};
			\draw(v1)--(v2);
			\draw(v1)--(v3);
			\draw(v2)--(v3);
			
			\draw(v2)--(v4);
			\draw(v4)--(v5);
			\draw(v4)--(v6);
		\end{tikzpicture}
		
	\end{subfigure}

	\begin{subfigure}{0.49\textwidth}
		\begin{tikzpicture}[scale=.15, transform shape]
			
		\node [draw, shape=circle,fill=black,scale=2] (v1) at  (0,0) {};
\node [draw, shape=circle,fill=black,scale=2] (v2) at  (10,0) {};
\node [draw, shape=circle,fill=black,scale=2] (v3) at  (5,9) {};

\node [draw, shape=circle,fill=black,scale=2] (v4) at  (20,0) {};
\node [draw, shape=circle,fill=black,scale=2] (v5) at  (24,6) {};
\node [draw, shape=circle,fill=black,scale=2] (v8) at  (30,6) {};
\node [scale=6] at (27,6) {\large $\cdots$}; {};
\node [draw, shape=circle,fill=black,scale=2] (v6) at  (-4,6) {};
\node [draw, shape=circle,fill=black,scale=2] (v7) at  (-10,6) {};
\node [scale=6] at (-7,6) {\large $\cdots$}; {};
			\node [scale=6] at (0,-1.5) {$x$};

\node [scale=6] at (10,-1.5) {$y$};	
\node [scale=6] at (20,-1.5) {$z$};		
			\draw(v1)--(v2);
			\draw(v1)--(v3);
			\draw(v2)--(v3);
			
			\draw(v2)--(v4);
			\draw(v4)--(v5);
			\draw(v4)--(v8);
			\draw(v1)--(v6);
			
			\draw(v1)--(v7);
		\end{tikzpicture}
		
	\end{subfigure}
	\begin{subfigure}{0.49\textwidth}
		\begin{tikzpicture}[scale=.15, transform shape]
			
		\node [draw, shape=circle,fill=black,scale=2] (v1) at  (0,0) {};
\node [draw, shape=circle,fill=black,scale=2] (v2) at  (10,0) {};
\node [draw, shape=circle,fill=black,scale=2] (v3) at  (5,9) {};

\node [draw, shape=circle,fill=black,scale=2] (v4) at  (20,0) {};
\node [draw, shape=circle,fill=black,scale=2] (v5) at  (30,0) {};
\node [draw, shape=circle,fill=black,scale=2] (v6) at  (34,6) {};
\node [draw, shape=circle,fill=black,scale=2] (v7) at  (40,6) {};
\node [scale=6] at (37,6) {\large $\cdots$}; {};
			\node [scale=6] at (10,-1.5) {$x$};

\node [scale=6] at (20,-1.5) {$y$};	
\node [scale=6] at (30,-1.5) {$z$};	
			\draw(v1)--(v2);
			\draw(v1)--(v3);
			\draw(v2)--(v3);
			
			\draw(v2)--(v4);
			\draw(v4)--(v5);
			\draw(v5)--(v6);
			\draw(v5)--(v7);
		\end{tikzpicture}
		
	\end{subfigure}
	\begin{subfigure}{0.49\textwidth}
		\begin{tikzpicture}[scale=.15, transform shape]
			
		\node [draw, shape=circle,fill=black,scale=2] (v1) at  (0,0) {};
\node [draw, shape=circle,fill=black,scale=2] (v2) at  (10,0) {};
\node [draw, shape=circle,fill=black,scale=2] (v3) at  (5,9) {};

\node [draw, shape=circle,fill=black,scale=2] (v4) at  (20,0) {};
\node [draw, shape=circle,fill=black,scale=2] (v5) at  (30,0) {};
\node [draw, shape=circle,fill=black,scale=2] (v6) at  (14,6) {};
\node [draw, shape=circle,fill=black,scale=2] (v7) at  (20,6) {};
\node [scale=6] at (17,6) {\large $\cdots$}; {};
			\node [scale=6] at (10,-1.5) {$x$};

\node [scale=6] at (20,-1.5) {$y$};	
\node [scale=6] at (30,-1.5) {$z$};			
			\draw(v1)--(v2);
			\draw(v1)--(v3);
			\draw(v2)--(v3);
			
			\draw(v2)--(v4);
			\draw(v4)--(v5);
			\draw(v2)--(v6);
			\draw(v2)--(v7);
		\end{tikzpicture}
		
	\end{subfigure}
	\begin{subfigure}{0.49\textwidth}
		\begin{tikzpicture}[scale=.15, transform shape]
		\node [draw, shape=circle,fill=black,scale=2] (v1) at  (0,0) {};
\node [draw, shape=circle,fill=black,scale=2] (v2) at  (10,0) {};
\node [draw, shape=circle,fill=black,scale=2] (v3) at  (5,9) {};

\node [draw, shape=circle,fill=black,scale=2] (v4) at  (20,0) {};
\node [draw, shape=circle,fill=black,scale=2] (v5) at  (30,0) {};
\node [draw, shape=circle,fill=black,scale=2] (v6) at  (14,6) {};
\node [draw, shape=circle,fill=black,scale=2] (v7) at  (20,6) {};
\node [scale=6] at (17,6) {\large $\cdots$}; {};
\node [draw, shape=circle,fill=black,scale=2] (v8) at  (34,6) {};
\node [draw, shape=circle,fill=black,scale=2] (v9) at  (40,6) {};
\node [scale=6] at (37,6) {\large $\cdots$}; {};
			\node [scale=6] at (10,-1.5) {$x$};

\node [scale=6] at (20,-1.5) {$y$};	
\node [scale=6] at (30,-1.5) {$z$};			
			\draw(v1)--(v2);
			\draw(v1)--(v3);
			\draw(v2)--(v3);
			
			\draw(v2)--(v4);
			\draw(v4)--(v5);
			\draw(v2)--(v6);
			\draw(v2)--(v7);
			\draw(v5)--(v8);
			\draw(v5)--(v9);
		\end{tikzpicture}
		
	\end{subfigure}
	\caption{$pc$-partitions of order $n-2$}\label{unicyclic6}
\end{figure}
\begin{proof}
	In Figure \ref{unicyclic6}, a $pc$-partition $\pi$ of order $n-2$ is illustrated for  graphs $G \in \bigcup_{i=1} ^{7} \mathcal{E}_i$, where $\pi$ consists of the set $\lbrace x,y,z\rbrace$ and $n-3$ singleton sets. Conversely,
let $\pi$ be a $PC(G)$-partition, and let $C =(v_1,v_2,v_3)$ be the cycle in $G$. We consider two cases.

\textbf{Case 1.} $\pi$ consists of two doubleton sets and $n-4$ singleton sets. If $n=5$, then the singleton set cannot form a paired coalition with any other sets, so we may assume that $n\geq 6$. Note that no singleton set can form a paired coalition with a doubleton set. Let $\lbrace v\rbrace$ and $\lbrace u\rbrace$ be two singleton sets in $\pi$ that are $pc$-partners. It follows that the edge $uv$ is a dominating edge in $G$, and so $\lbrace u,v\rbrace \cap V(C) \neq \emptyset$. Assume first that $uv \in E(C)$. Let $u=v_1$ and $v=v_2$.  Now we  define $A=N(u) \setminus \lbrace v_2,v_3\rbrace$ and $B=N(v) \setminus \lbrace v_1,v_3\rbrace$. Let $V_1, V_2 \subset \pi$, where $\lvert V_1\rvert =\lvert V_2\rvert =2$. Note that $E(G[V \setminus \lbrace v_1,v_2\rbrace])=\emptyset$, which contradicts the fact that $V_1$ and $V_2$ are $pc$-partners. Hence, $uv \notin E(C)$. Thus, we may assume, by symmetry, that $\lbrace u,v\rbrace \cap V(C) = \lbrace u\rbrace$. 
	Let $u=v_1$. Now we  define $A=N(u) \setminus \lbrace v\rbrace$ and $B=N(v) \setminus \lbrace u\rbrace$. Let $V_1, V_2 \subset \pi$, where $\lvert V_1\rvert =\lvert V_2\rvert =2$. Note that $E(G[A \cup B])= \lbrace v_2 v_3\rbrace$, which contradicts the fact that $V_1$ and $V_2$ are $pc$-partners. Hence, this case is impossible

\textbf{Case 2.} $\pi$ consists of a  set of order $3$ and $n-3$ singleton sets. An argument similar to the one used in the previous case shows that each singleton set must form a paired coalition with the set of order $3$. Define graph $G^\prime$, with $V(G^\prime)=V(G)$ and $E(G^\prime)=E(G)\setminus \lbrace v_1v_2 , v_1v_3  , v_2v_3\rbrace$, and let $T_1$,$T_2$ and $T_3$ denote the connected components of $G^\prime$ such that $v_1 \in V(T_1),$ $v_2 \in V(T_2)$ and $v_3 \in V(T_3)$.  Now we consider the following subcases.

\textbf{Subcase 2.1.} $\lbrace x,y,z\rbrace \cap V(C) =\emptyset$.
By Remark \ref{rem4}, $E(G[\lbrace x,y,z\rbrace]) \neq \emptyset$. By symmetry, we may assume that $xy \in E(T_1)$. Now if $z \in V(T_1)$, then we can find a vertex in $V(C)$ having no neighbor in $\lbrace x,y,z\rbrace$. Hence, we may assume, by symmetry, that $z \in V(T_2)$. Then, again we can find a vertex in $V(C)$ having no neighbor in $\lbrace x,y,z\rbrace$. Hence, this subcase is impossible.

\textbf{Subcase 2.2.} $\lvert \lbrace x,y,z\rbrace \cap V(C)\rvert =1$. By symmetry, we may assume that $z \in V(C)$. Let $z=v_1$. If $z$ has no neighbor in $\lbrace x,y\rbrace$, then $xy \in E(G)$. In this case, considering an arbitrary vertex $v \in V(G) \setminus \lbrace x,y,z\rbrace$ and a perfect matching $M$ in $G[\lbrace x,y,z,v\rbrace]$, we have $xy \in M$, implying that $N[z]=V(G)\setminus \lbrace x,y\rbrace$. Hence, $G \in \mathcal{E}_3 \cup \mathcal{E}_4 \cup \mathcal{E}_7$.  Otherwise, we may assume, by symmetry, that $xz \in E(G)$. Now if $y \notin V(T_1)$, then considering an arbitrary vertex $v \in V(G) \setminus \lbrace x,y,z\rbrace$ and a perfect matching $M$ in $G[\lbrace x,y,z,v\rbrace]$, we have $xz \in M$, implying that $N[y]=V(G)\setminus \lbrace x,z\rbrace$, which is impossible. Hence, $y \in V(T_1)$. Now considering a perfect matching $M$ in $G[\lbrace x,y,z,v_2\rbrace]$, we have $v_2 z \in M$, implying that $xy \in E(G)$. Further, it is easy to verify that $N(v_2)=N_C (v_2)$, $N(v_3)=N_C (v_3)$ and $N(x)=\lbrace y,z\rbrace$. Hence, $G\in \mathcal{E}_3 \cup \mathcal{E}_5 \cup \mathcal{E}_6 \cup \mathcal{E}_7$.

\textbf{Subcase 2.3.} $\lvert \lbrace x,y,z\rbrace \cap V(C)\rvert =2$. By symmetry, we may assume that $x \notin V(C)$. Let $\lbrace y,z\rbrace =\lbrace v_1,v_2\rbrace$. If $x$ has no neighbor in $\lbrace y,z\rbrace$, then considering an arbitrary vertex $v \in V(G) \setminus \lbrace x,y,z\rbrace$ and a perfect matching $M$ in $G[\lbrace x,y,z,v\rbrace]$, we have $yz \in M$, implying that $N[x]=V(G)\setminus \lbrace y,z\rbrace$. Note that $N(x)\setminus \lbrace v_3\rbrace \neq \emptyset$, for otherwise we would have $PC(G)=n$. Hence, $G \in \mathcal{E}_3$. Otherwise, we may assume, by symmetry, that $xy \in E(G)$. Then it is easy to verify that $N(v_3)=\lbrace y,z\rbrace$, and that $N(y)=\lbrace x,z,v_3\rbrace$. Hence, $G \in \mathcal{E}_1 \cup \mathcal{E}_3 \cup \mathcal{E}_4$.

\textbf{Subcase 2.4.} $V(C)=\lbrace x,y,z\rbrace$. It follows from Remark \ref{rem4} that each vertex in $V(G) \setminus V(C)$ must have a neighbor in $\lbrace x,y,z\rbrace$. Hence, $G \in \mathcal{E}_1 \cup \mathcal{E}_2$. This completes the proof.
\end{proof}
Using Theorems \ref{girth6} and \ref{girth5}, and Lemmas \ref{girth4} and \ref{girth3}, we have the following.
\begin{corollary}
	Let $G$ be a unicyclic graph of order $n$. Then $PC(G)=n-2$ if and only if $G \in (\bigcup_{i=1} ^{7} \mathcal{E}_1) \cup \mathcal{B} \cup \mathcal{B}_1 \cup \mathcal{B}_2 \cup \mathcal{D}_1 \cup \mathcal{D}_2 $.
\end{corollary}
	\section{Future research}
We conclude the paper with some research problems:
\begin{p}
	Characterize the trees admitting a paired coalition partition.
\end{p}
\begin{p}
	Characterize the trees attaining the bound of  Theorem \ref{tree-bound}.
\end{p}
\begin{p}
	Characterize all triangle-free graphs $G$ with $PC(G)=n-2$.
\end{p}
\begin{p}
A recent paper by Alikhani et al. discuss the concept of connected coalitions. It is intriguing to study paired Connected coalitions. What graphs admit a paired coalition partition that is also connected? That is, can we partition a graph into sets $\lbrace V_1,V_2,\dots ,V_k\rbrace$ such that neither of these sets are paired connected dominating sets, but $V_i$  can form a coalition with $V_j$  such that the set $V_i \cup V_j$  is; i) a dominating set, ii) a paired dominating set, and iii) the subgraph $G[V_i \cup V_j]$ is connected.	
\end{p}
\begin{p}
Can we study paired coalitions on cactuses?	
\end{p}
\begin{p}
Can we develop algorithms that can take a graph as input, and give you a paired coalition partition of maximum order $k$? Can these algorithms run in polynomial time, or which complexity classes do computations of paired coalition partitions on certain graph classes belong to?	
\end{p}

\begin{flushleft}
	\textbf{{\large Conflicts of interest}}\vspace{-3.5mm}
\end{flushleft}
The authors declare that they have no conflict of interest.

\end{document}